\documentclass[a4article,10pt]{article}
\usepackage{amsmath,amssymb,graphicx,bbm}
\usepackage{amsthm,verbatim}
\usepackage{mathrsfs,mathtools}
\usepackage[dvipsnames]{xcolor}
\usepackage{authblk}

\usepackage{setspace}
\usepackage{bm}

\usepackage[footnotesize,bf]{caption}
\usepackage[margin=2.9cm]{geometry}

\newcommand{\calA}{\mathcal{A}}
\newcommand{\vol}{\mathsf{vol}}

\newcommand{\bs}[1]{\boldsymbol{#1}}

\usepackage{array,multirow,booktabs}

\usepackage{mathdefs}

\usepackage[backend=bibtex,date=year,giveninits=true,sorting=nyt,natbib=true,maxcitenames=2,maxbibnames=10,url=false,doi=true,backref=false]{biblatex}
\renewbibmacro{in:}{\ifentrytype{article}{}{\printtext{\bibstring{in}\intitlepunct}}}

\usepackage[colorlinks,linkcolor = BrickRed, citecolor=blue]{hyperref} 

\definecolor{brickred}{rgb}{0.42, 0.22, 0.10}

\begin{document}
\title{Randomized weakly admissible meshes}
\date{}

\author[1,2]{\small Yiming Xu}
\author[1,2]{\small Akil Narayan}
\affil[1]{\small Department of Mathematics, University of Utah\\
       \texttt{yxu@math.utah.edu}}
\affil[2]{\small Scientific Computing and Imaging Institute, University of Utah\\
       \texttt{akil@sci.utah.edu}}
\renewcommand\Authands{and}
  
\maketitle

\begin{abstract}
A weakly admissible mesh (WAM) on a continuum real-valued domain is a sequence of discrete grids such that the discrete maximum norm of polynomials on the grid is comparable to the supremum norm of polynomials on the domain. The asymptotic rate of growth of the grid sizes and of the comparability constants must grow in a controlled manner. In this paper, we generalize the notion of a WAM to a hierarchical subspaces of not necessarily polynomial functions, and we analyze particular strategies for random sampling as a technique for generating WAM's. Our main results show that WAM's and their stronger variant, admissible meshes (AM's), can be generated by random sampling, and our analysis provides concrete estimates for growth of both the meshes and the discrete-continuum comparability constants.\end{abstract}

\begin{keyword}
Admissible meshes, Near-isometry, Norming sets, Random sampling, Weighted covering
\end{keyword}

\section{Introduction}

\subsection{Background}

Generating a discrete set to approximate a continuum is a task that arises in many areas of applied computational science. A concrete example is that of computing a so-called weakly admissible (discrete) mesh for polynomials. Given a compact domain $D \subset \R^d$ and a fixed $n \in \N$, consider a discrete set $\calA_n \subset D$ such that
\begin{align}\label{eq:linf-inequality}
  \sup_{x \in D} \left| p(x) \right| &\leq C_n \max_{x \in \calA_n} \left| p(x) \right|, & p \in P_n,
\end{align}
where $P_n$ is the subspace of algebraic $d$-variate polynomials of degree $n$ or less, and $C_n$ is some finite constant that may depend on $n$. With $N_n \coloneqq \dim P_n$, any sequence of meshes $\left\{\calA_n \right\}_{n=0}^\infty$ is called a \emph{weakly admissible mesh} (WAM) \cite{calvi_uniform_2008} if there exist absolute constants $a, b, C, C'<\infty$ such that $|\mathcal A_n|\leq CN_n^a$ and $C_n\leq C'N_n^b$ for all $n\in\N$. It is an \textit{admissible} mesh (AM) if $b=0$; see Definition \ref{def:Vwam} for more formal statements.
An ``optimal'' WAM would boast the smallest values of these exponents, $a = 1$, $b = 0$, since large $a$ implies increased growth of the mesh as $n$ increases, and large $b$ implies increased growth of the discrete-continuum equivalence constant as $n$ grows.
Admissible sets or meshes are closely related to norming sets, which provide a foundation for discretizations of continuum domains in function (in particular polynomial) approximation \cite{jetter1998norming,calvi_uniform_2008,kroo_optimal_2011,piazzon_optimal_2016}.

Note that the reverse inequality of \eqref{eq:linf-inequality} is straightforward (with constants $1$), so that a weakly admissible mesh yields equivalences between continuous and discrete supremum norms. This fact has been used to great effect in generating provably attractive meshes for polynomial approximation. However, it can be relatively difficult to generate such meshes for general domains $D$, and all existing constructions are essentially deterministic in nature. 

The purpose of this article is to show that random sampling (i.e., comprised of independent and identically-distributed samples) can be used to generate meshes that are weakly admissible or admissible, and to prescribe a minimal number of required samples so that one can achieve inequalities of the form \eqref{eq:linf-inequality} with high probability, and with constants $C_n$ whose large-$n$ behavior is known. We also provide a straightforward generalization of WAM's to non-polynomial spaces, and all our results apply in this case. 
For polynomial spaces and under regularity conditions on the set $D$, there are deterministic strategies for constructing AM's and WAM's \cite{kroo_optimal_2011,bos_geometric_2011,piazzon_optimal_2016}, which typically employ geometric constructions of meshes. An alternative strategy that we investigate in this paper is randomization, building meshes from random samples.  
When it is feasible for meshes to be randomly generated, one supposes that a large enough number of samples can be used to form WAM's. We provide quantitative analysis for such a procedure, deriving $n$-asymptotic quantities $a$ and $b$ for very general approximation spaces on general sets $D$.

\subsection{Contributions of this paper}

Our contributions in this article are twofold. We first generalize the notion of weakly admissible meshes to general hierarchical subspaces $\{V_n\}_{n=0}^\infty$, i.e., $V_n\subset V_{n+1}$ with $N_n:=\text{dim}(V_n)$. This generality allows us to consider generating meshes for approximation problems involving subspaces spanned by very general functions, not just polynomials. In particular, we seek to establish results analogous to \eqref{eq:linf-inequality}, centering around the inequality,
\begin{align}\label{eq:weak-wam}
  \sup_{x \in D} \left| v(x) \right| &\leq k_b N_n^b \max_{x \in \calA_n} \left| v(x) \right|, \hskip 10pt \forall v \in V_n, & |\mathcal{A}_n| &\leq k_a N_n^a.
\end{align}
The general algorithmic procedure we consider in this article is as follows: With $\{\rho_n\}_{n=0}^\infty$ some sequence of probability measures, let $\mathcal{A}_n$ be generated randomly as $|\mathcal{A}_n|$ i.i.d. samples from $\rho_n$ for each $n$. Our main results show that, for specific choices of $\rho_n$ and $|\mathcal{A}_n|$, this procedure generates weakly admissible meshes. A summary of these results is as follows: 

\begin{theorem}[Main results]
Let $\mu$ be a probability measure on $\R^n$ with closed support $D$ such that $\{V_n\}_{n=0}^\infty \subset L^2_{\mu}(D)$, and let $\epsilon > 0$. With $\omega$ an event with respect to the probability space corresponding to the product space of randomly drawn samples, there are finite constants $k_a = k_a(\epsilon)$, $k_b = k_b(\epsilon, \omega)$, and $N = N(\omega)$, such that the following statements are true with probability 1.
\begin{itemize}
  \item Let $\rho_n \equiv \mu$ for all $n$. With $a = q^\ast + \epsilon$ and $b = \frac{q^\ast}{2} + \epsilon$, then $\{\mathcal{A}_n\}_{n=0}^\infty$ forms an \textit{asymptotic} WAM, i.e., \eqref{eq:weak-wam} is true for all $n \geq N$; see Theorem \ref{thm:mu-wam}.
  \item Let $\rho_n = \mu_{V_n}$, where the latter is defined later in \eqref{eq:muV-def}. With $a = 1 + \epsilon$ and $b = \frac{q^\ast+1}{2} + \epsilon$, then $\{\mathcal{A}_n\}_{n=0}^\infty$ forms an asymptotic WAM; see Theorem \ref{thm:muV-WAM}.
  \item Assume $D$ is convex and compact, and that the elements of $V_n$ for each $n$ are smooth. Let $\rho_n = \nu_n$, where the latter is defined later in \eqref{am:nu}. With $a=p^\ast+\epsilon$ and $b = 0$, such $\{\mathcal{A}_n\}_{n=0}^\infty$ form an asymptotic AM; see Theorem \ref{thm:nun-AM}.
\end{itemize}
\end{theorem}

The constants $q^*$ and $p^*$ are defined in \eqref{eq:qast-def} and \eqref{p^*}, respectively.  
Dependence of $k_a$ and $k_b$ on $\epsilon$ is expected, reflecting a tradeoff in asymptotic sharpness of estimates versus proportionality constants in those asymptotics. Due to the level of generality of our statements, dependence of $k_b$ and $N$ on the event $\omega$ is unavoidable since one can construct $(\mu,V_0)$ such that for any fixed finite set $\mathcal{A}_0$ generated from random samples, there is positive probability that $\mathcal{A}_0$ is not a norming set for $V_0$ on $D$; see Example \ref{ex:haar-wavelets}.

However, such adversarial situations can be mitigated by assuming a type of regularity condition on the spaces $V_n$: If the sequence $\{V_n\}_{n=0}^\infty$ defines spaces such that the contour sets of every non-constant function in $V_n$ have $\rho_n$-measure 0, or $\mu$-measure 0 if $\rho_n\ll\mu$, then $N(\omega) = 0$ so that all the above statements are true for all $n\in\N$ with the same $(a,b)$; see Corollaries \ref{cor:mu-wam}, \ref{cor:muV-WAM}, \ref{cor:nun-AM}. Hence, under this zero-measure contour condition, the sets $\mathcal{A}_n$ are genuine WAM's and AM's with probability 1. In particular, when $\rho_n$ has a Lebesgue density and $V_n$ are polynomial spaces, this zero-measure contour condition is satisfied; see Proposition \ref{eq:muzc-analytic}.

The main focus of this article is to establish asymptotic \textit{rates} of WAM behavior constructed from random samples, exemplified by the constants $p^\ast$ and $q^\ast$ which are deterministic and independent of $\epsilon$:
 
$q^*$ is the asymptotic log-ratio of the Bernstein-Markov factor of the space $V_n$ and $N_n$; see \eqref{eq:qast-def}, and $p^*$ roughly measures the growth rate of a weighted covering number of $D$ relative to $N_n$; see \eqref{p^*}.  Both $q^*$ and $p^*$ depend only on the prescribed measure $\mu$ and the hierarchical spaces $V_n$.  The sampling measure $\mu_{V_n}$ introduced above was utilized in \cite{cohen_optimal_2016} to construct least squares approximations via random sampling.  Our methodology can be used to generate WAM's for hierarchical polynomial spaces more exotic than the total degree spaces $P_n$ as originally stated in \eqref{eq:linf-inequality}, but due to the generality of our results the first two bullet points above are weaker than some existing results for polynomials and deterministically-constructed meshes, e.g., in \cite{calvi_uniform_2008}. We will elaborate on these statements in Section \ref{2} and \ref{sec:weighted-spaces}. 

The rest of this article is organized as follows. In Section \ref{2}, we introduce the notation and definitions for weakly admissible meshes. A few examples illustrating the definitions are provided in Section \ref{sec:weighted-spaces} to facilitate understanding. 
In Section \ref{3}, we borrow the idea of the near-isometry property of random matrices to obtain a sampling measure that generates weakly admissible meshes with exponents $a=q^*+\epsilon$ and $b=\frac{q^\ast}{2} + \epsilon$, and then appropriately reweight the measure to make $a$ near-optimal at mild cost on $b$: $a=1+\epsilon$ and $b=\frac{q^\ast+1}{2} + \epsilon$. 
In Section \ref{wcs}, we introduce the concept of weighted covering and investigate some properties associated with it. 
In Section \ref{4}, we introduce a novel sampling strategy to generate admissible meshes, and our analysis lies in a probabilistic argument using so-called weighted covering numbers.

\begin{table}[h!]
  \begin{center}
  \resizebox{\textwidth}{!}{
    \renewcommand{\tabcolsep}{0.4cm}
    \renewcommand{\arraystretch}{1.3}
    {\scriptsize
    \begin{tabular}{@{}cp{0.8\textwidth}@{}}
      \toprule
      $\mu$,$D$ & probability measure on $\R^d$ and $D = \mathrm{supp}\; \mu$ \\
      $V$ & $N$-dimensional subspace of functions in $L^2_\mu(D)$ \\
      $\mu_V$ & the $V$-induced measure on $D$ \\
      $K_\mu(V)$, $\lambda(x)$ & the Bernstein-Markov factor for $V$ in $L^2_\mu(D)$, and the normalized $V$-Christoffel function \\
      $R_n$& square root of the sum of Christoffel functions induced by gradients \\
      $\nu_n$& the $R_n$-induced measure (weighted by $R_n^d$) \\
      $W$ & $N$-dimensional space of functions in $V$ weighted by the Christoffel function \\
      $v_i$, $w_i$ & any orthonormal basis for $V$ in $L^2_\mu$, and for $W$ in $L^2_{\mu_V}$, respectively \\
      $\mathcal{A}$, $M$ & A finite set in $D$ of size $M$ \\
      $\bs{A}_{\mathcal{A}}$ & $M \times N$ Vandermonde-like matrix associated with $v_n$ on $\mathcal{A}$\\
      $\|\cdot\|_{2,\mu}$, $\|\cdot\|_{\infty}$ & $L^2_\mu(D)$ norm, and $L^\infty(D)$ norm, respectively \\
      $\|\cdot\|_{2,\mathcal{A}}$, $\|\cdot\|_{\infty,\mathcal{A}}$ & Discrete $\ell^2$ and $\ell^\infty$ norms, respectively, on $\mathcal{A}$\\
       $\vol(A)$, $A\subset\R^d$ & Lebesgue measure of $A$ in $\R^d$\\
    \bottomrule
    \end{tabular}
  }
    \renewcommand{\arraystretch}{1}
    \renewcommand{\tabcolsep}{12pt}
  }
  \end{center}
  \caption{Notation used throughout this article.}\label{tab:notation}
\end{table}

\section{Notation}\label{2}
\subsection{Function spaces}
Let $d \in \N$ be fixed. Consider a probability measure $\mu$ on $\R^d$ whose support is denoted $D \subseteq \R^d$. We do not require any particular conditions on $D$ at present, but the results in Section \ref{4} regarding admissible meshes require compactness of $D$.
The Hilbert space $L^2_{\mu}(D;\C)$ is endowed with the inner product and norm
\begin{align*}
  \left\langle f, g \right\rangle_\mu &\coloneqq \int_{\R^d} f(x) \bar{g}(x) \dx{\mu}(x),& \left\|f \right\|^2_{2,\mu} &\coloneqq \left\langle f, f \right\rangle_\mu.
\end{align*}
The supremum norm on $D$ for functions $f$ is 
\begin{align*}
  \left\| f \right\|_{\infty} &= \sup_{x \in D} \left| f(x) \right|.
\end{align*}
Note $D$ may be unbounded. If $\mathcal{A} \subset D$ is a size-$M$ set of points, we define discrete $L^2$ and $L^\infty$ norms as
\begin{align*}
  \left\| f \right\|^2_{2,\calA} &\coloneqq \frac{1}{M} \sum_{x \in \calA} |f(x)|^2, & 
  \left\| f \right\|_{\infty,\calA} &\coloneqq \max_{x \in \calA} \left|f(x)\right|.
\end{align*}

Let $V$ be an $N$-dimensional subspace of functions in $L^2_\mu(D)$. 
We can choose an associated orthonormal basis $v_1, \ldots, v_N \in V$. Let $\delta_x$ denote the $x$-centered Dirac delta distribution. For any $x \in D$, the $V$-valued function
\begin{align*}
  K\left(x, \cdot\right) = \sum_{i=1}^N \bar{v}_i(x) v_i(\cdot)
\end{align*}
is the $V$-Riesz representor of $\delta_x$ in $L^2_\mu$. For any $v \in V$, we have 
\begin{align*}
  \left| v(x) \right| = \left|\left\langle v, K\left(x, \cdot\right) \right\rangle_\mu \right| \leq \left\| v \right\|_{2,\mu} \left\| K\left(x, \cdot\right) \right\|_{2,\mu} = \left\| v \right\|_{2,\mu} \sqrt{\sum_{i=1}^N |v_i(x)|^2 }.
\end{align*}
And so,
\begin{align}\label{eq:l2-linfty}
  \left\| v \right\|_{2,\mu}^2 \leq \left\| v\right\|^2_{\infty} &\leq K_{\mu}(V) \left\| v \right\|^2_{2,\mu}, &
  K_{\mu}(V) &\coloneqq \left\| K(x,x) \right\|_\infty,
\end{align}
where the lower inequality above holds since $\mu$ is a probability measure. Throughout this article we assume that $K_\mu(V)$ is finite, which excludes, e.g., algebraic polynomials on unbounded domains. The optimal (smallest) value of the equivalence factor $K_\mu(V)$ is $N$:
\begin{align*}
  K_\mu(V) = \left\| \sum_{i=1}^N |v_i(x)|^2 \right\|_\infty \geq \sum_{i=1}^N \left\| v_i(x)\right\|^2_{2,\mu} = N.
\end{align*}
If $V$ is chosen as a space of polynomials up to a certain degree, selecting $\mu$ as a so-called optimal measure achieves this optimal factor \cite{bloom_convergence_2010}. Assuming $V$ contains differentiable functions, we also define
\begin{align}
&R(x) = \left(\sum_{i=1}^N\|\nabla v_i(x)\|^2\right)^{1/2}&R_\mu(V) = \|R(x)\|_\infty,\label{R}
\end{align}
where $\|\cdot\|$ is the Euclidean norm on vectors. 
This function will play a role in our analysis involving covering numbers for admissible meshes.
In that discussion, we assume that $R(x)$ is strictly positive on $D$.

Finally, note in the above definitions, both $K(x,x)$ and $R(x)$ are independent of the choice for orthonormal basis $v_1, \cdots, v_N$ and in particular depend only on $(\mu,V)$. 

\subsection{WAM's for general hierarchical subspaces}
We generalize the notion of admissible meshes to general hierarchical subspaces. 
Let $\left\{V_n \right\}_{n=0}^\infty$ denote any sequence of finite-dimensional hierarchical subspaces, i.e., $V_n \subset V_{n+1}$ and $\dim V_n < \dim V_{n+1}$ for all $n$. We set 
\begin{align*}
  N_n \coloneqq \dim V_n,
\end{align*}
which is a sequence of strictly increasing positive integers.

\begin{definition}[Asymptotic weakly admissible meshes for hierarchical spaces]\label{def:Vwam}
 Let $D$ be a closed set in $\R^d$. Consider $\left\{\calA_n\right\}_{n=0}^\infty$ a collection of finite subsets of $D$. Assume there is a collection of constants $\{C_n\}_{n=0}^\infty$, $N\in\N$ and $a, b, k_a, k_b<\infty$, such that for all $n\geq N$, 
  \begin{itemize}
    \item $\left\| v \right\|_{\infty} \leq C_n \|v\|_{\infty, \mathcal A_n}$, $\forall v \in V_n$
    \item $C_n \leq k_b N_n^b$
    \item $\left| \calA_n \right|\leq k_a N_n^a$
  \end{itemize}
  Then $\left\{\calA_n\right\}_{n=0}^\infty$ is called an \textit{asymptotically weakly admissible mesh}; it is called an \textit{asymptotically admissible mesh} if $b=0$. 
  In particular, when $N = 0$, an asymptotically weakly admissible mesh is the same as a (classical) weakly admissible mesh (WAM). If both $b=N=0$, an asymptotically admissible mesh coincides with a (classical) admissible mesh (AM).
\end{definition}

Compared to the original definition of WAM's \cite{calvi_uniform_2008}, Definition \ref{def:Vwam} allows to consider more general hierarchical spaces other than polynomials. 
On the other hand, the asymptotic WAM/AM portions of Definition \ref{def:Vwam} are weaker in the sense that they only require the $\ell_\infty$-norm comparability condition hold for all sufficiently large $n$.  
For random sampling methods, such a compromise is required due to the generality of the statement, as illustrated by the following example:
\begin{example}[Weaker definitions are necessary for random sampling]\label{ex:haar-wavelets}
  Let $\dx{\mu}(x) = \dx{x}$ on the unit interval $D = [0,1]$. Suppose that $V_n$ contains a nonzero function $v$ such that $\mu\{v^{-1}(0)\}>0$. This may happen, for instance, when $V_n$ consists of compactly supported wavelets. In this case, for every $n\in\N$, $M>N_n$ and $\mathcal A = \{X_1, \cdots, X_M\}$ that is drawn as i.i.d. samples according to $\mu$, 
  \begin{align*}
 \Pr\left[ \|v\|_{\infty, \mathcal A} = 0\right] = \left(\mu\{v^{-1}(0)\}\right)^M>0.
  \end{align*}
  Consequently, a strict W/AM condition in our definition (i.e., holding for all $n \geq N$) in general does not imply the $\ell_\infty$-norm comparability condition for all $n\in\N$.  Such examples illustrate the need to consider \textit{asymptotic} W/AM statements.
\end{example}

\subsection{$\mu$ZC sequences}
While Example \ref{ex:haar-wavelets} shows the need for weaker (asymptotic) notions of WAMs, under appropriate assumptions on $V_n$, asymptotic W/AM's are equivalent to classical W/AM's for randomly constructed meshes with probability $1$. The requisite assumption on the spaces $V_n$ precisely disallows the situation in Example \ref{ex:haar-wavelets} where functions can vanish on sets with positive $\mu$-measure. 

\begin{definition}[$\mu$ZC sequences]
  Let $\mu$ be a probability measure on $D \subset \R^d$. A sequence of hierarchical subspaces $\{V_n\}_{n \geq 0}$ with domain $D$ is a $\mu$-measure zero contour ($\mu$ZC) sequence if, for all $n\in\N$,
  \begin{align}\label{eq:muzc}
    \mu\{v^{-1}(0)\} &= 0 & \forall v&\in V_n\backslash\{0\}.
  \end{align}
\end{definition}
Such sequences are not difficult to produce. For example, if $\mu$ has a Lebesgue density then there are many sequences that are $\mu$ZC sequences.

\begin{proposition}\label{eq:muzc-analytic}
  Let $\mu$ be a probability measure on $D \subset \R^d$. If $\mu$ is absolutely continuous with respect to the Lebesgue measure on $\R^d$ and $\{V_n\}_{n=0}^\infty$ is a hierarchical sequence of subspaces containing only real-analytic functions, then $\{V_n\}_{n=0}^\infty$ is a $\mu$ZC sequence.
\end{proposition}
\begin{proof}
  Assume the contrary, that there is some $n \in \N$ and $v \in V_n \backslash \{0\}$ such that
  \begin{align}\label{eq:non-muzc}
    \mu(v^{-1}(0)) > 0 \hskip 5pt \Longrightarrow \hskip 5pt \vol(v^{-1}(0)) > 0.
  \end{align}
  Since $v$ is nonzero and real analytic, the zero set of $v$ has Lebesgue measure $0$, i.e., $ \vol(v^{-1}(0)) = 0$; see \cite{Mityagin_2020} for a short elementary proof. 
This is a contradiction to \eqref{eq:non-muzc}. Therefore, $\{V_n\}_{n=0}^\infty$ is a $\mu$ZC sequence.
\end{proof}
The proposition above implies, in particular, that if $\{V_n\}_{n=0}^\infty$ is a hierarchical sequence of polynomial subspaces with $\mu$ any probability measure having a Lebesgue density, then $\{V_n\}_{n=0}^\infty$ is a $\mu$ZC sequence. Our main result in this section demonstrates the utility of $\mu$ZC sequences: \textit{Asymptotic} W/AM's of $\mu$ZC sequences are (classical) W/AM's.

\begin{theorem}\label{thm:WAM-con}
  Let $\{V_n\}_{n=0}^\infty$ be a sequence of hierarchical subspaces on domain $D$. For $n \geq 0$, let $\mathcal{A}_n$ be a set of $|\mathcal{A}_n|$ iid samples from a probability measure $\rho_n$. Assume there is a probability measure $\mu$ such that $\rho_n \ll \mu$ for all $n$, and that $\{V_n\}_{n=0}^\infty$ is a $\mu$ZC sequence. If $\{\mathcal{A}_n\}_{n=0}^\infty$ is an asymptotic W/AM with exponent parameters $(a, b)$ with probability 1, then it is also a (classical) W/AM with parameters $a, b$ with probability 1.
\end{theorem}

Before proving this result, we need an intermediate fact that generalizes a result from \cite{Bass_2005}.
\begin{lemma}\label{lemma:muzc-det}
  Assume $\{V_n\}_{n=0}^\infty$ is a $\mu$ZC sequence of hierarchical subspaces, and let $\mathcal{A}_n$ be generated as i.i.d. samples of $\rho_n$, with $\rho_n \ll \mu$ for all $n \geq 0$, i.e., 
  \begin{align*}
    \mathcal{A}_n &= \{X_{n,1}, \ldots, X_{n,|\mathcal{A}_n|} \}, & X_{n,j} \stackrel{\textrm{i.i.d.}}{\sim} \rho_n.
  \end{align*}
  Define the square alternant matrices
  \begin{align*}
    \bs{V}_n &\in \C^{N_n \times N_n}, & (\bs{V}_n)_{i, j} &= v_j(X_{n,i}).
  \end{align*}
  Then $\mathrm{Pr}[\det \bs{V}_n = 0] = 0$ for all $n \geq 0$.
\end{lemma}
\begin{proof}
Fix $n\geq 0$. 
Let $v_1, \cdots, v_{N_n}$ be a basis in $V_n$.
The matrix $\bs{V}_n$ is given by 
\begin{align*}
\bm V_n = \begin{pmatrix}
  v_1(X_1) & v_{2}(X_1)&\cdots &v_{N_n}(X_1)\\
  v_1(X_2)& v_{2}(X_2)&\cdots & v_{N_n}(X_2)\\
\vdots&\vdots&\ddots & \vdots\\
  v_1(X_{N_n})& v_{2}(X_{N_n})&\cdots & v_{N_n}(X_{N_n})
\end{pmatrix}
.
\end{align*}
  Let $D_{n,\ell}$ denote the determinant of the upper-left $\ell \times \ell$ block of $\bs{V}_n$ (i.e., the size-$\ell$ upper-left principal minor of $\bs{V}_n$). We seek to show that $\mathrm{Pr}[D_{n,N_n} = 0] = 0$. 
Since $\{V_n\}_{n=0}^\infty$ is a $\mu$ZC, we have $\mu\left(v_1^{-1}(0)\right) = 0$, which, in tandem with $\rho_n \ll \mu$, shows that $D_{n,1} = v_1(X_{n,1}) = 0$ occurs with probability 0, i.e., $\mathrm{Pr}[D_{n,1} = 0] = 0$.

Proceeding by induction, fix $1 \leq \ell < N_n$, and assume $\mathrm{Pr}[D_{n,\ell} = 0] = 0$.
Write $\bm V_{n,\ell+1}$ as the following block form:
\begin{align*}
\bm V_{n,\ell+1} = \begin{pmatrix}
\bm V_{n,\ell} &\bm w_{\ell+1}\\
  \bm u^T_{\ell+1}(X_{\ell+1})& v_{\ell+1}(X_{\ell+1})
\end{pmatrix},
\end{align*}
where 
\begin{align*}
  &\bm u_{\ell+1}(X_{\ell+1}) = (v_{1}(X_{\ell+1}), \cdots, v_{\ell}(X_{\ell+1}))^T&\bm w_{\ell+1} = (v_{\ell+1}(X_1), \cdots, v_{\ell+1}(X_\ell))^T.
\end{align*}
Note that $\bm u_{\ell+1}$ depends only on $X_{\ell+1}$, and $\bm w_{\ell+1}$ depends only on $X_1, \cdots, X_\ell$. In particular, the inductive hypothesis implies that $\bs{V}_{n,\ell}$ is invertible with probability 1. On this probability-1 event, an exercise with Schur complements for computing determinants of block matrices yields
\begin{align*}
  D_{n,\ell+1} = D_{n,\ell} \left( v_{\ell+1}(X_{n,\ell+1}) - \bs{u}_{\ell+1}^T \bs{V}_{n,\ell}^{-1} \bs{w}_{\ell+1} \right).
\end{align*}
Thus, $D_{n,\ell+1}$ vanishes on the event $\{D_{n,\ell} \neq 0\}$ if and only if,
\begin{align*}
  v_{\ell+1}(X_{\ell+1})- \bs{u}_{\ell+1}^T(X_{\ell+1}) \bs{V}_{n,\ell}^{-1} \bs{w}_{\ell+1} = 0 \Longleftrightarrow F(X_{\ell+1}; X_1, \cdots, X_\ell) = 0,
\end{align*}
where 
\begin{align*}
F(x; X_1, \cdots, X_\ell) = v_{\ell+1}(x) - (v_1(x), \cdots, v_\ell(x))\bm V_{n,\ell}^{-1}\bm w_{\ell+1}\in V_n.
\end{align*}
It is clear that $F\not\equiv 0$ for any $(X_1, \ldots, X_{\ell})$ satisfying $D_{n,\ell} \neq 0$ since the coefficient of $v_{\ell+1}$ is $1$ and $v_{\ell+1}$ is independent of $v_1, \cdots, v_\ell$.  
Thus, with the Fubini-Tonelli theorem we conclude
\begin{align*}
\mathrm{Pr}\left[ \{D_{n,\ell+1} = 0\}\right] &\leq \mathrm{Pr}\left[\{D_{n,\ell} = 0\}\right] + \mathrm{Pr}\left[ \{D_{n,\ell+1} = 0\} \;\cap\; \{D_{n,\ell} \neq 0\} \right]\\
 &\stackrel{\text{induction}}{=} 0+\int_{D_{n,\ell} \neq 0} \rho_{n+1}\left(F^{-1}(0; x_1, \ldots, x_\ell)\right) \dx{\rho_n(x_1)} \cdots \dx{\rho_n(x_\ell)} \\
  &\stackrel{\eqref{eq:muzc}, \rho_{n+1}\ll\mu}{=} \int_{D_{n,\ell} \neq 0} 0\, \dx{\rho_n(x_1)} \cdots \dx{\rho_n(x_\ell)}=0.
\end{align*}
Thus, $\mathrm{Pr}\left[ D_{n,\ell+1} = 0 \right]  = 0$ as desired.
\end{proof}

We are now in a position to prove Thereom \ref{thm:WAM-con}.

\begin{proof}[Proof of Thoerem \ref{thm:WAM-con}]
Fix $n\in\N$, and let $\calA_n=\{X_1, \cdots, X_M\}$ where $M = |\calA_n|\geq N_n$.  
For any $v\in V_n$ with $\|v\|_{2,\mu}=1$, write $v = \sum_{i\in [N_n]}\alpha_iv_i(x)$ with $\|\alpha\|_2=1$, where $\alpha = (\alpha_1,\cdots, \alpha_{N_n})$. 
By \eqref{eq:l2-linfty}, $\|v\|_{\infty,\mu}\leq K_\mu(V_n)<\infty$. 
On the other hand, 
\begin{align*}
\|v\|_{\infty,\calA_n}\geq\|v\|_{\infty,\{X_1, \cdots, X_{N_n}\}} = \|\bm V_n\alpha\|_\infty\geq\frac{1}{\sqrt{N_n}}\|\bm V_n\alpha\|_2\geq\frac{1}{\sqrt{N_n}}\sigma_{\min}(\bm V_n),
\end{align*}
where $\sigma_{\min}(\bm V_n)$ is the smallest singular value of $\bm V_n$. 
Hence, 
\begin{align}
&\|v\|_{\infty,\mu}\leq \frac{K_\mu(V_n)\sqrt{N_n}}{\sigma_{\min}(\bm V_n)}\|v\|_{\infty,\calA_n}& 0\neq v\in V_n.\label{mygood}
\end{align}
  Since $\{V_n\}_{n \geq 0}$ is a $\mu$ZC sequence, then Lemma \ref{lemma:muzc-det} implies $\bm V_n$ is invertible (i.e. $\sigma_{\min}(\bm V_n)>0$) with probability $1$ for each $n$.
This combined with \eqref{mygood} implies that $\sup_{0\neq v\in V_n}\frac{\|v\|_{\infty,\mu}}{\|v\|_{\infty,\calA_n}}<\infty$ with probability $1$ for fixed $n$. 
Now take a union bound over $n$ to conclude
\begin{align}
\Pr\left[\sup_{0\neq v\in V_n}\frac{\|v\|_{\infty,\mu}}{\|v\|_{\infty,\calA_n}}<\infty, \ \forall n\in\N\right] = 1.\label{happy event}
\end{align}
Under our assumption, we also have 
\begin{align}
\Pr\left[\text{$\calA$ is an asymptotic WAM for $\{V_n\}_{n=0}^\infty$ with parameters $a, b$}\right] = 1.\label{happier event}
\end{align}
Therefore, the intersection of the probabilistic events in \eqref{happy event} and \eqref{happier event} occurs with probability $1$. 
Now take a sample $\omega$ from this intersected probability-1 event: There exists $N(\omega)\in\N$ and $\kappa_a(\omega), \kappa_b(\omega)<\infty$, such that for all $n \geq N(\omega)$, 
\begin{align}
\left\| v \right\|_{\infty} &\leq \kappa_b N_n^b \|v\|_{\infty, \mathcal A_n}&\forall v \in V_n\label{lam}\\
\left| \calA_n \right| &\leq \kappa_a N_n^a\nonumber.
\end{align}
Define 
\begin{align*}
  k_a(\omega) &\coloneqq \max\left\{ \kappa_a(\omega), \max_{j=0, \ldots, n-1}\frac{|\mathcal{A}_j|}{N_j^a}\right\} \\
  k_b(\omega) &\coloneqq \max\left\{ \kappa_b(\omega), \max_{j=0, \ldots, n-1} \frac{1}{N_j^b} \sup_{0 \neq v \in V_j} \frac{\|v\|_{\infty,\mu}}{\|v\|_{\infty,\calA_n}} \right\},
\end{align*}
both of which are finite with probability 1 due to \eqref{happy event}. Then we have that $\mathcal{A}$ is a classical WAM with parameters $(a,b)$, and is a classical AM if $b=0$.
\end{proof}

We have established that randomly generated asymptotic W/AM's are in fact (classical) W/AM's for $\mu$ZC sequences if the sampling measures $\rho_n$ are absolutely continuous with respect to $\mu$. The remainder of this paper therefore focuses on proving asymptotic W/AM properties for randomly generated meshes; all the sampling measures we employ satisfy $\rho_n \ll \mu$.

\section{Weighted spaces and induced probability measures}\label{sec:weighted-spaces}

The $L^2_\mu$-$L^\infty$ equivalence established by \eqref{eq:l2-linfty} can be improved to the optimal equivalence if one considers weighted spaces: For a finite-dimensional $V$ with $L^2_\mu$-orthonormal basis $v_i$, define the ($L^2$) Christoffel function
\begin{align*}
  \lambda(x) = \lambda_{V,\mu}(x) = \frac{N}{K(x,x)} = \frac{N}{\sum_{i=1}^N |v_i(x)|^2},
\end{align*}
and consider the associated space of weighted elements from $V$:
\begin{align}\label{eq:W-def}
  W \coloneqq \sqrt{\lambda(x)} V \coloneqq \left\{ \sqrt{\lambda} v \; \big| \; v \in V \right\}.
\end{align}
We also define a weighted measure $\mu_V$ via
\begin{align}\label{eq:muV-def}
  \dx{\mu_V}(x) \coloneqq \frac{1}{\lambda(x)} \dx{\mu}(x) = \frac{1}{N} K(x,x) \dx{\mu}(x),
\end{align}
which is another probability measure on $D$ that is absolutely continuous with respect to $\mu$, i.e., $\mu_V \ll \mu$. The functions $w_i = v_i \sqrt{\lambda(x)}$ are an $L^2_{\mu_V}(D)$-orthonormal basis for $W$, and 
\begin{align*}
  K_{\mu_V}\left( W\right) = \left\| \sum_{i=1}^N |w_i(x)|^2 \right\|_\infty = N,
\end{align*}
so that $\mu_V$ is an optimal measure for $W$, and we have the optimal equivalence relation 
\begin{align}\label{eq:optimal-l2-linf}
  \left\| w \right\|_{2,\mu_V} &\leq \left\| w \right\|_\infty \leq N \left\| w \right\|_{2,\mu_V}, & w &\in W.
\end{align}
The measure $\mu_V$ has utility in recent computational strategies for constructing discrete least-squares approximations \cite{cohen_optimal_2016}. In this article, we will call $\mu_V$ the $V$-induced measure for $\mu$. The term ``induced" stems from historical context: For certain $\mu$ and polynomial spaces $V$, the measure $\mu_V$ is an additive mixture of tensor-product measures; the univariate measures that define the tensor-product measures in this case are similar to induced orthogonal polynomials \cite{gautschi_set_1993}. Sampling from such non-standard measures is computationally efficient and feasible by exploiting properties of orthogonal polynomials \cite{narayan_computation_2017}.

\subsection{Polynomial spaces}
We will sometimes be concerned with the special case when $V$ is a subspace of polynomials. In this specialized case we denote the space $P$ as an $N$-dimensional space of polynomials. It is convenient (but not necessary) to use multi-indices to define these spaces. 

We let $\alpha \in \N^d$ denote a $d$-dimensional multi-index and use $\Sigma \subset \N^d$ to denote a finite set of multi-indices. Associated to any $\Sigma$, we define the subpsace of algebraic polynomials spanned by monomials:
\begin{align}\label{eq:Psigma}
  P_\Sigma \coloneqq \mathrm{span} \left\{ x^\alpha \; |\; \alpha \in \Sigma \right\}.
\end{align}
A particularly special set of multi-indices are those corresponding to the total-degree space of polynomials:
\begin{align}\label{eq:Sigman}
  \Sigma_n &= \left\{ \alpha \in \N^d \; | \; |\alpha| \leq n \right\}, & n \in \N.
\end{align}
We will use the abbreviation $P_n \coloneqq P_{\Sigma_n}$.

Finally, we note that there are many finite-dimensional polynomial spaces that cannot be written in the form \eqref{eq:Psigma}. This is a deficiency in our presentation style in that we emphasize the specific class of subspaces \eqref{eq:Psigma}. However, all our theoretical results extend to general polynomial subspaces.

Setting $V_n=P_n$, Definition \ref{def:Vwam} for an asymptotic WAM with $N=0$ is consistent with the one used in \cite{bos_geometric_2011}. The optimal value of the WAM exponent $b$ is $b = 0$, and such meshes are known to exist for domains exhibiting a polynomial Markov inequality \cite{calvi_uniform_2008}, although the construction relies on grids that achieve certain fill distances and can thus be cumbersome for sufficiently complex domains using a deterministic approach. 
The optimal value of the exponent $a$ in general is $a = 1$ since the inequality
\begin{align}\label{eq:bernstein-markov-total-degree}
  \left\| p \right\|_{\infty} &\leq C_n \|p\|_{\infty,\mathcal A_n}, & p &\in V_n
\end{align}
holds for some finite $C_n$ only if $\calA_n$ is determining for $P_n$ (i.e., for any $p \in V_n$, $p(x) = 0$ for all $x \in \calA$ implies $p \equiv 0$). 

Note that in the modified definition $D$ is allowed to be unbounded, meaning that WAM's for $V_n$ are only sensible if $V_n$ contains functions that are bounded on $D$. 
This subtlety will be reiterated when we discuss random sampling for generating admissible meshes. 
In the remainder of this paper, we will refer to the above definition when speaking of a WAM. 

\subsection{The constant $q^\ast$}
We will make a further assumption on the spaces $V_n$ that involves the $L^2_\mu$ machinery we have introduced, namely that they satisfy
\begin{align}\label{eq:qast-def}
  q^\ast \coloneqq q^\ast\left(\mu, \left\{V_n\right\}_{n=0}^\infty \right) = \limsup_{n\rightarrow \infty} \frac{\log K_\mu(V_n)}{\log N_n} < \infty.
\end{align}
Note that since $K_\mu(V_n) \geq \dim V_n = N_n$, then $q^\ast \geq 1$. Our main results using concentration of measure characterize the WAM exponents $a$ and $b$ via proportionality to $q^\ast$, and thus small $q^\ast$ is desirable. Requiring finite $q^\ast$ can be related to similar notions in the polynomial context. If we choose $V_n = P_n$, then finite $q^\ast$ along with
\begin{align}\label{eq:total-degree-dimension}
  |\calA| \geq N_n = \dim P_n = \left(\begin{array}{c} n+d \\ d \end{array}\right) \sim \frac{n^d}{d!}
\end{align}
and the fact the $N_n^{1/n} \rightarrow 1$ implies
\begin{align*}
  \lim_{n \rightarrow \infty} K_\mu(P_n)^{1/n} = 1,
\end{align*}
showing that the pair $(D, \mu)$ satisfies the so-called Bernstein-Markov property. Pairs that satisfy the Bernstein-Markov property have fundamental connections to various results in approximation theory, and we refer to \cite[Section 5]{bloom_polynomial_2012} for a detailed summary. In fact, if $\mu$ satisfies a ``density condition" then it is known that $q^\ast$ is finite for sequences of fairly general polynomial subspaces \cite[Corollary 4.2.2]{hussung_pluripotential_2020}. Thus, our requirement that $q^\ast < \infty$ is not unnatural, but is slightly stronger than a Bernstein-Markov property when specialized to polynomials.

To illustrate values of $q^\ast$, we summarize three special choices for $\mu$ and $V_n$.
\begin{example}[Complex exponentials]\label{ex:complex-exp}
  Let $\dx{\mu}(x) = \dx{x}$ on the unit cube $D = [0,1]^d$ for arbitrary $d \geq 1$. For an arbitrary subset $F \subset \Z^d$ of size $N$ given by $F = \left\{ f_1, \ldots, f_N \right\}$, define
    \begin{align}\label{eq:vn-complex-exp}
      v_n(x) &= \exp\left[2 \pi i (f_n \cdot x) \right], & V &= \mathrm{span} \left\{ v_n \right\}_{n=1}^N,
    \end{align}
    where $f_n \cdot x$ is the standard componentwise inner product between two elements in $\R^d$. Then $v_n$ is an orthonormal basis for $V$, and $K_\mu(V) = N$, and thus we can choose any hierarchical collection of subspaces $V_n \subset V_{n+1}$ defined by corresponding hierarchical sets $F_n \subset F_{n+1}$. We then have
    \begin{align*}
      q^\ast = \limsup_{n\rightarrow \infty} \frac{\log K_\mu(V_n)}{\log N_n} = \limsup_{n\rightarrow \infty} \frac{\log N_n}{\log N_n} = 1,
    \end{align*}
    thus achieving the optimal $q^\ast$ factor. In this case we also have $\mu_V = \mu$. In the language of \cite{bloom_convergence_2010} for polynomials, the measure $\mu$ is an optimal measure.
\end{example}

\begin{example}[Tensor-product Jacobi polynomials]\label{ex:jacobi}
  Let $\dx{\mu}(x) \propto \prod_{j=1}^d \left(1-x^{(j)}\right)^\alpha \left(1-x^{(j)}\right)^\beta$ on $x \in [-1,1]^d = D$ and $\alpha, \beta \in \N$. Choosing $V_n = P_n$, the space of $d$-variate polynomials of degree $n$ or less, an orthonormal family for $V_n$ is provided by tensorized Jacobi polynomials. The estimates in \cite[Theorem 9]{migliorati_multivariate_2015} show that $K_\mu(V_n) \leq N_n^{2(\gamma+1)}$, where $\gamma = \max\left\{ \alpha, \beta \right\}$. Therefore, 
    \begin{align*}
      q^\ast = \limsup_{n\rightarrow \infty} \frac{\log K_\mu(V_n)}{\log N_n} \leq \limsup_{n\rightarrow \infty} \frac{2(\gamma+1)\log N_n}{\log N_n} = 2(\gamma+1).
    \end{align*}
  Here, while $N_n$ and $K_\mu(V_n)$ both grow exponentially in $d$, the quantity $q^\ast$ does not.
\end{example}

\begin{example}[Tensor-product Chebyshev polynomials]\label{ex:cheb}
  With $\mu$ and $V_n$ as in the previous example, now take $\alpha = \beta = -1/2$, so that an orthonormal basis is provided by tensorized Chebyshev polynomials. Univariate Chebyshev polynomials $T_k(y)$, $y \in [-1,1]$, satisfy $T^2_k(y) \leq 2$ for all $k$, so that $K_\mu(V_n) \leq N_n 2^d$. Thus, 
    \begin{align*}
      q^\ast = \limsup_{n\rightarrow \infty} \frac{\log K_\mu(V_n)}{\log N_n} \leq \limsup_{n\rightarrow \infty} \frac{d \log 2  + \log N_n}{\log N_n} = 1.
    \end{align*}
    Since $q^\ast \geq 1$ always holds, we conclude that $q^\ast = 1$. Note that the bound $K_\mu(V_n) \leq N_n 2^d$ holds when $V_n = P_{\Sigma}$ for any multi-index set $\Sigma$. Thus, the behavior $q^\ast = 1$ holds for very general hierarchical polynomial spaces. This suggests that the Chebyshev measure $\mu$ is $n$-asymptotically optimal.
\end{example}

\section{Randomized weakly admissible meshes}\label{3}
In this section we prove two of our main results, showing that particular random sampling strategies generate WAM's with exponents $(a,b)$ that depend linearly on $q^\ast$.

\subsection{Discrete randomized near-isometries}
The main strategy for our approach comes in two parts: First we generate a finite mesh $\calA$ that emulates the $L^2_\mu$ norm on $V$. We subsequently use that mesh and the $L^2$-$L^\infty$ equivalence relations described earlier in order to transform comparability of $\|\cdot\|_{2,\mu}$ and $\|\cdot\|_{2,\mathcal{A}}$ into comparability between $\|\cdot\|_{\infty}$ and $\|\cdot\|_{\infty,\mathcal{A}}$. The first part of this strategy, the construction of $\calA$ based on $L^2_\mu$ properties, is the subject of this section.

With $\{v_i\}_{i=1}^N$ an orthonormal basis for $V$, we require the following algebraic formulation: the $M \times N$ matrix $\bs{A}_{\calA}$ has entries
\begin{align*}
  \left(\bs{A}_{\calA}\right)_{m,i} &= \frac{1}{\sqrt{M}} v_{i}\left(x_m\right), & \calA &= \left\{x_m\right\}_{m=1}^{M}, \ \ (m,i)\in [M]\times [N].
\end{align*}
The problem of finding a discrete mesh capable of emulating the $L^2_\mu$ norm is conceptually identical to finding a stable discrete least-squares problem defined by the matrix $\bs{A}$. We codify this in the following theorem from \cite{cohen_stability_2013} with a more explicit constant.
\begin{theorem}[\cite{cohen_stability_2013}]\label{thm:ls}
  Let $\calA = \left\{ X_m \right\}_{m=1}^M$, where the $X_m$ are independent and identically distributed draws of a random variable distributed according to the probability measure $\mu$. For any $r > 0$ and $0 < \delta < 1$, suppose that 
  \begin{align}\label{eq:sample-count}
    \frac{M}{\log M} \geq \frac{3(1+r)}{\delta^2} K_\mu(V).
  \end{align}
Then 
  \begin{align}\label{eq:gramian-proximity}
    \mathrm{Pr} \left[ \left\| \bs{A}_\calA^* \bs{A}_{\calA} - \bs{I} \right\| \geq \delta \right] \leq 2 M^{-r},
  \end{align}
  where $\|\cdot\|$ is the induced (spectral) norm on matrices.
\end{theorem}
The quantity $\bs{A}_{\calA}^* \bs{A}_\calA$ is the Gramian of the basis $v_i$ with respect to the discrete inner product $\langle\cdot,\cdot\rangle_{2,\calA}$, and thus \eqref{eq:gramian-proximity} quantifies the $L^2$ proximity of a discrete measure supported on $\calA$ to $\mu$ on the space $V$. The sample count complexity \eqref{eq:sample-count} couples $M$ and $N$ with dependence on the proximity parameter $\delta$ and the success parameter $r$, but also involves the Bernstein-Markov factor $K_\mu(V)$. When $V$ is a polynomial space defined by multi-index set $\Sigma$, many standard continuous probability measures yield extremely large $K_\mu(V)$, often depending exponentially on $d$ and algebraically on the maximum polynomial degree in $\Sigma$ \cite{chkifa_discrete_2015}.

The authors in \cite{hampton_coherence_2015,narayan_christoffel_2016} note that introducing weights related to $\lambda(x)$ into the least-squares algorithm would result in a procedure with optimal (minimal) sample count. The authors in \cite{cohen_optimal_2016} propose a computationally feasible procedure for drawing samples from the induced measure $\mu_V$ using this weighting idea, and arrive at the following result:
\begin{theorem}[\cite{cohen_optimal_2016}]\label{thm:optimal-ls}
  Let $\calA = \left\{ X_m \right\}_{m=1}^M$, where the $X_m$ are independent and identically distributed draws of a random variable distributed according to the probability measure $\mu_V$. Introduce weights $\omega_m \coloneqq \lambda_V(x)$, and define the diagonal $M \times M$ matrix $\bs{W}$ with entries $(W)_{m,m} = \omega_m$. With $r$ and $\delta$ as in Theorem \ref{thm:ls}, assume
  \begin{align}\label{eq:christoffel-sample-count}
    \frac{M}{\log M} \geq \frac{3(1+r)}{\delta^2} N.
  \end{align}
  Then 
  \begin{align}
    \mathrm{Pr} \left[ \left\| \bs{A}_\calA^* \bs{W} \bs{A}_{\calA} - \bs{I} \right\| \geq \delta \right] \leq 2 M^{-r}.\label{good}
  \end{align}
\end{theorem}
Note that the sample complexity \eqref{eq:christoffel-sample-count} is near-optimal, up to the $\log M$ factor.

If a mesh $\calA$ satisfies the conditions of either Theorem \ref{thm:ls} or \ref{thm:optimal-ls}, then we can establish equivalences between discrete and continuous $L^2$ norms.
\begin{corollary}\label{cor:norm-equivalence}
  \begin{subequations}\label{eq:norm-equivalence}
  \begin{enumerate}
    \item Suppose $\calA$ is a mesh satisfying the conditions of Theorem \ref{thm:ls}. Then with probability at least $1 - 2 M^{-r}$,
      \begin{align}\label{eq:norm-equivalence-p}
        \frac{1}{1+\delta} \left\| v \right\|^2_{2,\calA} &\leq \left\| v \right\|^2_{2,\mu} \leq \frac{1}{1-\delta} \left\| v \right\|^2_{2,\calA}, & v &\in V
      \end{align}
    \item Let $\calA$ be a mesh satisfying the conditions of Theorem \ref{thm:optimal-ls}. Then with probability at least $1 - 2 M^{-r}$,
      \begin{align}\label{eq:norm-equivalence-q}
        \frac{1}{1+\delta} \left\| w \right\|^2_{2,\calA} \leq \left\| w \right\|^2_{2,\mu_V} &\leq \frac{1}{1-\delta} \left\| w \right\|^2_{2,\calA},& w &\in W,
      \end{align}
      where $W$ is the space \eqref{eq:W-def} of $\sqrt{\lambda}$-weighted $V$ functions.
  \end{enumerate}
  \end{subequations}
\end{corollary}
\begin{proof}
 We prove the second statement; the proof of the first statement is similar. 
 For an arbitrary $w \in W$, represent $w(x) = \sum_{i=1}^N u_i w_i(x)$ according to the $L^2_{{\mu_V}}$-orthonormal basis $\left\{ w_i\right\}_{i=1}^N$, so that $\left\| w \right\|_{2, \mu_{V}} = \left\| \bs{u} \right\|$ for $\|\cdot\|$ the Euclidean norm on vectors. Also, 
  \begin{align*}
    \left( \sqrt{\bs{W}} \bs{A}_{\mathcal{A}} \bs{u} \right)_m = \sum_{i=1}^N \sqrt{\frac{\lambda\left(X_m\right)}{M}} u_i w_i \left(X_m\right) = \frac{1}{\sqrt{M}} \sum_{i=1}^N u_i w_i\left(X_m\right) = \frac{1}{\sqrt{M}} w\left(X_m\right).
  \end{align*}
  Thus, $\left\| w \right\|_{2,\calA} = \left\| \sqrt{\bs{W}} \bs{A}_{\mathcal{A}} \bs{u}\right\|$. Assuming the complement of the probabilistic event in \eqref{good},  
  \begin{align*}
    \left\| w \right\|^2_{2,\calA} = \left\| \sqrt{\bs{W}} \bs{A}_{\mathcal{A}} \bs{u}\right\|^2 &= \bs{u}^* \left( \bs{I} + \left( \bs{A}^*_{\calA} \bs{W} \bs{A}_{\calA} - \bs{I} \right) \right) \bs{u} \\
                                                                                                 & \leq \left(1 + \delta\right)\left\| \bs{u} \right\|^2  = \left(1 + \delta\right) \left\| w\right\|^2_{2,\mu_V}.
  \end{align*}
  This establishes the lower inequality in \eqref{eq:norm-equivalence-q}. The upper inequality is shown in the same way:
  \begin{align*}
    \left\| w \right\|^2_{2,\calA} = \left\| \sqrt{\bs{W}} \bs{A}_{\mathcal{A}} \bs{u}\right\|^2 &= \bs{u}^* \left( \bs{I} - \left( \bs{I} - \bs{A}^*_{\calA} \bs{W} \bs{A}_{\calA} \right) \right) \bs{u} \\
                                                                                                 & \geq \left(1 + \delta\right)\left\| \bs{u} \right\|^2 = \left(1 - \delta\right) \left\| w\right\|^2_{2,\mu_V}.
  \end{align*}
\end{proof}

\subsection{Sampling from $\mu$}

Since $V$ is a subspace of $L^2_\mu$, it is reasonable to believe that taking a large number of i.i.d. samples from $\mu$ will eventually allow one to approximate the $L^\infty$ norm of any element in $V$.  
Theorem \ref{thm:V-linf-bounds} below shows that for a fixed subspace $V$, i.i.d. samples yield an equivalence relation between the discrete and continuous maximum norms.
\begin{theorem}\label{thm:V-linf-bounds}
  Let $V$ be a given subspace of dimension $N$, and assume that $M$ is large enough to satisfy \eqref{eq:sample-count} for some $\delta \in (0,1)$ and $r > 0$. Let $\mathcal{A}$ have size $M$ with elements comprised of i.i.d. samples from $\mu$. Then, with probability $1 - 2 M^{-r}$, we have
  \begin{align}\label{eq:mu-i.i.d.-wam}
    \left\| v \right\|_\infty &\leq \sqrt{\frac{K_\mu(V)}{1 - \delta}} \left\| v \right\|_{\infty, \mathcal{A}}, & v &\in V.
  \end{align}
\end{theorem}
\begin{proof}
  For any $v \in V$, we have 
  \begin{align}\label{eq:muV-linf-bounds}
    \left\| v \right\|^2_\infty &\stackrel{\eqref{eq:l2-linfty}}{\leq} K_\mu(V) \left\| v \right\|^2_{2, \mu} \stackrel{\eqref{eq:norm-equivalence-p}}{\leq} \frac{K_\mu(V)}{1 - \delta} \left\| v \right\|^2_{2, \mathcal{A}} \leq \frac{K_\mu(V)}{1 - \delta} \left\| v \right\|^2_{\infty, \mathcal{A}},
  \end{align}
  where the second inequality holds with probability $1 - 2 M^{-r}$.
\end{proof}
Note that this theorem is suboptimal: not only do we require $M/N$ to scale like $K_\mu(V)$, but we also pay a penalty factor of $K_\mu(V)$ in the norm comparability result \eqref{eq:mu-i.i.d.-wam}. Nevertheless, we can use this construction to form asymptotic weakly admissible meshes: Theorem \ref{thm:V-linf-bounds} together with the Borel-Cantelli lemma yields the following result.

\begin{theorem}\label{thm:mu-wam}
  Let $\left\{V_n \right\}_{n=0}^\infty$ be given. For each $n$, define 
  \begin{align*}
    \calA_n &\coloneqq \left\{X_{m} \right\}_{m=1}^{M_n}, & X_{m} \stackrel{\text{i.i.d.}}{\sim} \mu.
  \end{align*}
If $M_n = 25q^* K_\mu(V_n) \log N_n$, then with probability 1,  $\left\{\calA_n\right\}_{n=0}^\infty$ forms an asymptotic WAM for $\{V_n\}_{n=0}^\infty$ and $\mu$ with exponents $a = q^\ast + \tau$ for any $\tau > 0$, and $b = \frac{q^\ast}{2} + \tau$.
\end{theorem}

\begin{proof}
We use the abbreviation $K_n \coloneqq K_\mu(V_n)$ to reduce notational clutter. 
Set $\delta = 1/2$. Our assumptions ensure that
\begin{align*}
\liminf_{n\to\infty}\frac{M_n}{K_n \log M_n}&=\liminf_{n\to\infty}\frac{25q^*\log N_n}{\log K_n +\log (25q^*\log N_n)}\\
& = \liminf_{n\to\infty}\frac{25q^*\log N_n}{\log K_n}\stackrel{\eqref{eq:qast-def}}{=}\frac{25q^*}{q^*}=25> 12(1+r)
\end{align*}
for $1<r<13/12$, so that for sufficiently large $n$, 
\begin{align*}
\frac{M_n}{\log M_n}\geq 12(1+r)K_n = \frac{3(1+r)}{\delta^2}K_n.
\end{align*}
The above condition verifies that \eqref{eq:sample-count} is satisfied. It follows from Theorem \ref{thm:V-linf-bounds} that for any $\tau>0$ and $1<r<13/12$, with probability $1-2 M_n^{-r}$,  
\begin{align}
 \left\| v \right\|_\infty &\leq \sqrt{2K_\mu(V)} \left\| v \right\|_{\infty, \mathcal{A}}\leq C(\tau)N_n^{\frac{q^*}{2}+\tau} \left\| v \right\|_{\infty, \mathcal{A}}, & v &\in V,\label{eq:V-bounds}
\end{align}
where $C(\tau)>0$ is some constant depending only on $\tau$. 
In this case, 
\begin{align*}
\limsup_{n\to\infty}\frac{\log M_n}{\log N_n} = \frac{\log K_n}{\log N_n} = q^*. 
\end{align*}
Thus, there exists a sufficiently large constant $H(\tau)>0$ such that
\begin{align*}
&|\mathcal A_n| = M_n\leq H(\tau)N_n^{q^*+\tau}&\forall n\in\N. 
\end{align*}
Define $E_n$ as the probabilistic event that the above inequality holds. Since  
  \begin{align*}
\sum_{n=0}^{\infty} \mathrm{Pr}[E_n^c]\leq 2\sum_{n=0}^\infty M_n^{-r}\leq 2\sum_{n=0}^\infty \min\{n^{-r},1\} < \infty,
  \end{align*}
by the Borel-Cantelli Lemma, the probability that $E_n^c$ happens infinitely often is 0.
That is, outside a null set, for any realization $\omega$, there exists a sufficiently large $N(\omega)$, such that
\begin{align}
& \left\| v \right\|_\infty \leq C(\tau)N_n^{\frac{q^*}{2}+\tau} \left\| v \right\|_{\infty, \mathcal{A}_n(\omega)}&\forall n>N(\omega).\label{abv}
\end{align}
Thus, by Definition \ref{def:Vwam}, $\{\mathcal{A}_n\}_{n=0}^\infty$ forms an asymptotic WAM with $a = q^\ast + \tau, b = \frac{q^\ast}{2} + \tau$, i.e., $k_a = H(\tau), k_b = C(\tau)$. 
\end{proof}

By combining the above result with Theorem \ref{thm:WAM-con}, we can produce a (classical) WAM.
\begin{corollary}\label{cor:mu-wam}
  Assume the conditions of Theorem \ref{thm:mu-wam}. If in addition $\{V_n\}_{n=0}^\infty$ is a $\mu$ZC sequence, then with probability 1, $\{\mathcal{A}_n\}_{n=0}^\infty$ is a (classical) WAM with exponents $a = q^\ast + \tau$ and $b = \frac{q^\ast}{2} + \tau$ for any $\tau > 0$.
\end{corollary}

\subsection{Sampling from $\mu_V$}
Meshes generated by randomly sampling from $\mu$ are suboptimal, as shown above. The WAM exponents of such meshes are effectively $a = q^\ast$ and $b = \frac{q^\ast}{2}$. We can entirely remove the dependence of $a$ on $q^\ast$ by considering weighted meshes. This section essentially repeats the computations of the previous section, but by replacing $\mu$ with $\mu_V$ and $V$ with $W$. Since the proofs are almost identical to the ones in the previous section, we omit them for brevity. 
\begin{theorem}\label{thm:q-grid}
  With $V$ given, let $\left\{ X_m \right\}_{m=1}^M$ be a sequence of i.i.d. random variables distributed according to $\mu_V$. Assume that $M$ satisfies \eqref{eq:christoffel-sample-count} for some $r > 0$ and $\delta < 1$. Then with probability at least $1 - 2 M^{-r}$, 
  \begin{align}
    \left\| w \right\|_{\infty} &\leq \sqrt{\frac{N}{1 - \delta}} \left\| w \right\|_{\calA, \infty}, & w &\in W.\label{eq:weighted-rederive}
  \end{align}
\end{theorem}

The WAM's described above provide ways to bound supremum norms of functions in the weighted space $W$.
We can translate these back into estimates on the space $V$ by paying a mild penalty factor.
\begin{theorem}
  Let $\mu$, $D$, and $V$ be given, and assume that $1 \in V$. Let $\left\{ X_m \right\}_{m=1}^M$ be i.i.d. samples from $\mu_V$. If, for some $r > 0$ and $0 < \delta < 1$, $M$ is large enough to satisfy \eqref{eq:christoffel-sample-count}, then with probability at least $1 - 2 M^{-r}$, 
  \begin{align*}
    \left\| v \right\|_{\infty} &\leq \sqrt{N} \sqrt{\frac{K_\mu(V)}{1 - \delta}} \left\| v \right\|_{\calA, \infty}, & v &\in V.
  \end{align*}
\end{theorem}
\begin{proof}
  Since $V$ contains constant functions, then let $v_1 \equiv 1$ be a particular choice of the first element in an orthonormal basis for $V$. Then we have
  \begin{align*}
   \lambda(x) =\frac{N}{\sum_{n=1}^N v_n^2} \leq N.
  \end{align*}
  Any $w \in W$ can be written as $\sqrt{\lambda} v$ for some $v \in V$, and so 
  \begin{subequations}\label{eq:thm-temp-1}
  \begin{align}
    \left\| w \right\|_\infty \geq \left\| v \right\|_\infty \inf_{x \in D} \lambda(x) \geq \left\| v \right\|_\infty \sqrt{\frac{N}{K_\mu(V)}}.
  \end{align}
  Likewise, 
  \begin{align}
    \left\| w\right\|_{\infty,\calA} \leq \left\| \sqrt{\lambda} \right\|_{\infty,\calA} \left\| v \right\|_{\infty,\calA} \leq \sqrt{N} \left\| v \right\|_{\infty,\calA}.
  \end{align}
  \end{subequations}
  Chaining \eqref{eq:weighted-rederive} with relations \eqref{eq:thm-temp-1} proves the theorem.
\end{proof}
Finally, we can generate a WAM using the result above:
\begin{theorem}\label{thm:muV-WAM}
  Let $\mu$, $D$, and $\left\{V_n \right\}_{n=0}^\infty$ be given, and assume that $1 \in V_0$. For each $n$, define 
  \begin{align*}
    \calA_n &\coloneqq \left\{X_{m} \right\}_{m=1}^{M_n}, & X_{m} \stackrel{\text{i.i.d.}}{\sim} \mu_{V_n}.
  \end{align*}
  Assume that $M_n = 25 N_n \log N_n$. Then $\left\{\calA_n\right\}_{n=0}^\infty$ forms an asymptotic weakly admissible mesh for $\{V_n\}_{n=0}^\infty$ with exponents $a = 1 + \tau$ for any $\tau > 0$, and $b = \frac{q^\ast + 1}{2}+\tau$.
\end{theorem}
 
Since $\mu_{V_n} \ll \mu$ for any $n$, then combining the above result with Theorem \ref{thm:WAM-con}, produces a (classical) WAM.
\begin{corollary}\label{cor:muV-WAM}
  Assume the conditions of Thereom \ref{thm:muV-WAM}. If in addition $\{V_n\}_{n=0}^\infty$ is a $\mu$ZC sequence, then with probability 1 $\{\mathcal{A}_n\}_{n=0}^\infty$ is a (classical) WAM with exponents $a = 1 + \tau$ and $b = \frac{q^\ast+1}{2} + \tau$ for any $\tau > 0$.
\end{corollary}

\section{Weighted coverings}\label{wcs}

Designing an AM, a stronger WAM with $b=0$, via random sampling relies on generating grids with good space-filling properties. This section provides some infrastructure for space-filling designs that we will use to generate AM's via random sampling, and hence plays an analogous role to Section \ref{sec:weighted-spaces} for our design of WAM's.  
Our results for AM's require stronger assumptions on the domain $D$ and subspaces $V_n$: Throughout this section, we assume that $D$ is compact and that the subspaces $V_n$ contain continuously differentiable functions for all $n\in\N$. Note that these extra assumptions allow polynomial approximation as a specialization.

A baseline for randomized space-filling designs is the uniform sampler on $D$, and the corresponding sufficient sampling size can be obtained by analyzing the covering number of the domain. In the deterministic context, low discrepancy sequences \cite{dick_digital_2010} are a standard approach for ``uniformly" filling a volume with points. However, this approach is more difficult when $D$ is not a hypercube, and so we will investigate randomized approaches via sampling.  To understand how points generated from an arbitrary measure fill the domain, we introduce the following definition of weighted covering:   

\begin{definition}[$f$-weighted covering]\label{am:def1}
Let $D\subset \R^d$ be compact and $f: D\to\R_+$ be a continuous function, where $\R_+ = (0, \infty)$. 
For $r>0$, let $B_r(x)=\{z\in\R^d: \|z-x\|_2\leq r\}$. 
Let $m_f = \min_{u\in D} f(u)>0$. 
For $y\in D$ and $r>0$, define 
\begin{align*}
F_r(y) = \frac{\min_{z\in B_r(y)\cap D}f(z)}{m_f}.
\end{align*}
A set $\mathcal N\subset D$ is called an $f$-weighted $\epsilon$-covering of $D$ if 
\begin{align}\label{eq:reps-def}
&D\subset\bigcup_{y\in\mathcal N}B_{r_y}(y)&r_y\leq r(y, \epsilon) :=\max_{c\geq 0}\min\left\{\epsilon F_c(y), c\right\}.  
\end{align}
In particular, when $f$ is a constant function, $F_r(y)\equiv 1$, so that an $f$-weighted $\epsilon$-covering is the same as an $\epsilon$-covering.
\end{definition}
 
The general idea behind an $f$-weighted $\epsilon$-covering that motivates the definition of $r$ in \eqref{eq:reps-def} is 
to cover points in $D$ using balls of radius proportional to $f$, with the location(s) $\argmin_{u\in D}f(u)$ covered by balls of radius $\epsilon$. 
In particular, we refer to the quantity 
\begin{align*}
&F_0(y) = \frac{f(y)}{m_f}\cdot\epsilon&y\in D
\end{align*}
as the \emph{local radius} at $y$. 
Under this definition, points within the same ball in a covering may have different local radius. 
To address this, we require that the local radius of any point in the ball must be at least $r$, where $r$ the radius of the ball. Simultaneously, we wish to choose $r$ as large as possible (hence the max-min condition in \eqref{eq:reps-def}).

A property of the weighted coverings introduced above that we exploit is that, as $\epsilon \downarrow 0$, one can obtain better covering numbers compared to a standard $\epsilon$-covering. To see why, first note that since $F_r(y)\geq 1$ for any $y\in D$ and $r>0$, then $r(y, \epsilon)\geq \epsilon$.  
This implies that any $\epsilon$-covering is also an $f$-weighted $\epsilon$-covering. 
Also, for fixed $\epsilon$, $F_c(y)$ is a continuous and non-increasing function of $c$ which evaluates to $1$ for $c\geq\text{diam}(D)$. 
Therefore, 
\begin{align}
r(y, \epsilon)=\epsilon F_{r(y, \epsilon)}(y).\label{nice}
\end{align}
It follows from \eqref{nice} that $r(y,\epsilon)$ is non-increasing in $\epsilon$.
As an immediate consequence, if $0<\epsilon_1<\epsilon_2\leq 1$, then
\begin{align*}
\frac{r(y, \epsilon_1)}{r(y,\epsilon_2)}=\frac{\epsilon_1F_{r(y, \epsilon_1)}(y)}{\epsilon_2F_{r(y, \epsilon_2)}(y)}\geq\frac{\epsilon_1}{\epsilon_2}.
\end{align*}
This implies that the radius of a weighted covering ball at a given point scales slower than $\epsilon$, and this in turn leads to a potentially better covering number rate constant as $\epsilon\to 0$.

In our particular case, we will make the choice of weight,
\begin{align*}
f(x) = R_n(x)^{-1}.
\end{align*}  
where we recall that $R_n(x)$ in given in \eqref{R} and associated to the subspace $V = V_n$.
If $V_n$ is a subspace of continuously differentiable functions, then such an $f(x)$ will satisfy the assumption in Definition \ref{am:def1}. 
The $f$-weighted covering number will be used to analyze the sampling properties of an $R_n$-weighted probability measure on $D$, which is defined as
\begin{align}
&\dx{\nu_n(x)} = \frac{R_n(x)^d}{\int_D R_n(x)^d \dx{x}}\dx{x} & x\in D.\label{am:nu}
\end{align}
Note that \eqref{am:nu} is well-defined as long as $R_n(x)\in L^d(D)$. 
The measure $\nu_n(x)$ plays a crucial role in the design of sampling for admissible meshes.

\subsection{The constant $p^\ast$}
Let $\mathcal S_{f, \epsilon}(D)$ be the set of $f$-weighted $\epsilon$-coverings of $D$.  
For any $\mathcal N\in\mathcal S_{f, \epsilon}(D)$, define
\begin{align} 
G_\mathcal N = \max_{y\in\mathcal N}\left(\frac{\max_{x\in B_{r(y, \epsilon)}\cap D}f(x)}{\min_{x\in B_{r(y, \epsilon)}\cap D}f(x)}\right)^d.\label{myG}
\end{align}
Note for every $y\in\mathcal N$, its covering radius satisfies $r_y\leq r(y, \epsilon)$.
When the strict inequality holds, elongating $r_y$ to make it equal to $r(y, \epsilon)$ will still result in an $f$-weighted $\epsilon$-covering by Definition \ref{am:def1}.
Thus, $G_\mathcal N$ evaluates the change of magnitude of $f$ for every $y\in\mathcal N$ in the most ``conservative" sense.
A quantity that will appear in our analysis for the AM exponents is
\begin{align}
&p^* = \limsup_{n\to\infty}\frac{\log\left(\inf_{\mathcal N\in\mathcal S_{f, \epsilon_n}(D)}|\mathcal N|G_\mathcal N\right)}{\log N_n}& \epsilon_n = (3R_\mu(V_n))^{-1} = \frac{1}{3\max_{x\in D} R_n(x)}\label{p^*}.
\end{align}
The constant $3$ above is non-essential and can be replaced with any constant greater than $2$ in the analysis in Section \ref{4}.  
Similar to $q^*$ for WAM's, $p^*$ will characterize the exponents AM's in our result.  
Thus, it is essential for $p^*$ to be finite to have practical interest. 
We next demonstrate that, provided the geometry of the level sets of $R_n(x)$ are ``regular'', $p^*$ is finite if $\int_DR_n(x)^d\dx{x}$ has polynomial growth as $n\to\infty$. In particular we emphasize Corollary \ref{cor:pfinite}, which ensures a finite $p^\ast$ under mild assumptions on $R_n$ and $D$.

\subsection{Finiteness of $p^\ast$}
We first introduce some extra notation. Let 
\begin{align}
\delta_n = \frac{1}{3\min_{x\in D} R_n(x)},\label{mydelta}
\end{align}
which measures the maximum local radius for points in $D$ when the minimizer of $f$ has local radius $\epsilon_n$.  
For any $\gamma>1$, define
\begin{align*}
T_{n,\gamma}: = \left\lceil\log_\gamma\left(\frac{\delta_n}{\epsilon_n}\right)\right\rceil = \left\lceil\log_\gamma\left(\frac{\max_{x\in D} R_n(x)}{\min_{x\in D} R_n(x)}\right) \right\rceil,
\end{align*}
and we subsequently partition $D$ based on level sets of $R_n$ as 
\begin{align*}
&D = \bigcup_{\ell=1}^{T_{n,\gamma}}A^{(n)}_{\ell,\gamma}& A^{(n)}_{\ell,\gamma} := \left\{x\in D: \gamma^{-\ell}\cdot\max_{x\in D} R_n(x)\leq R_n(x)\leq \gamma^{-\ell+1}\cdot\max_{x\in D} R_n(x)\right\}.
\end{align*}
The particular partition strategy that we employ will allow $\gamma$ to depend on $n$. Under technical assumptions associated with this type of partition, we can bound $p^*$ as follows:
\begin{theorem}\label{thm:new}
  Assume there exist two sequences $\{\gamma_n\}_{n=0}^\infty$ ($\gamma_n>1$) and $\{w_n\}_{n=0}^\infty$ such that 
\begin{enumerate}
\item Both $\gamma_n$ and $w_n$ have at most polynomial growth with respect to $N_n$:
\begin{align*}
  \zeta  &\coloneqq \limsup_{n\to\infty}\frac{\log \gamma_n}{\log N_n}<\infty, & 
  \kappa &\coloneqq \limsup_{n\to\infty}\frac{\log w_n}{\log N_n}<\infty.
\end{align*}
\item The level sets $A^{(n)}_{\ell, \gamma_n}$ are ``well-separated'', i.e., there exists an integer $q<\infty$ such that for every $n \in \N$ and for any $\ell'$ with $\ |\ell'-\ell| > q$, 
\begin{align*}
B_{\gamma_n^{\ell}\epsilon_n}(x)\cap A^{(n)}_{\ell',\gamma_n} &= \varnothing&\forall x\in A^{(n)}_{\ell,\gamma_n}.
\end{align*}
\item The (standard) covering number of $A^{(n)}_{\ell,\gamma_n}$ is $w_n$-comparable to its theoretical lower bound, i.e., 
\begin{align*}
  \mathsf N(A^{(n)}_{\ell,\gamma_n}, \gamma_n^{\ell-q-1}\epsilon_n)&\leq w_n\cdot\frac{\vol(A^{(n)}_{\ell,\gamma_n})}{\vol(B_{\gamma_n^{\ell-q-1}\epsilon_n}(0))}&\ell\in [T_{n,\gamma_n}], \ n \in \N,
\end{align*}
where $\mathsf N(\cdot, \epsilon)$ denotes the standard $\epsilon$-covering number of a set.   
\end{enumerate}
Then, 
\begin{align}
p^*\leq \kappa + (q+1)d\zeta + \limsup_{n\to\infty}\frac{\log\left(\int_D R_n(x)^d\dx{x}\right)}{\log N_n}.\label{p^*bdd}
\end{align}
\end{theorem}

Before going to the proof, we first provide situations when the assumptions described are met. First, if
\begin{align*}
\limsup_{n\to\infty}\frac{\log\left(\frac{\delta_n}{\epsilon_n}\right)}{\log N_n} = \limsup_{n\to\infty}\frac{\log\left(\frac{\max_{x\in D}R_n(x)}{\min_{x\in D}R_n(x)}\right)}{\log N_n}<\infty, 
\end{align*}
then one can take $\gamma_n = \frac{\delta_n}{\epsilon_n}$ to satisfy the first part of assumption 1, and this also yields satisfaction of assumption 2 since with this choice of $\gamma_n$, the multi-scale requirement associated to $q>0$ can be achieved with a single-scale, so that $q=0$.

Assumption 3 and the second part of Assumption 1 can be satisfied if $A^{(n)}_{\ell,\gamma_n}$ is a union of almost disjoint regular sets which are not ``too small". We articulate this result next.
\begin{lemma}\label{mylemma}
Assume that $R_n$ is continuous on $D$. 
Suppose that the $A^{(n)}_{\ell,\gamma_n}$ defined in Theorem \ref{thm:new} can be written as a union of almost disjoint subsets, i.e., there exist $L_{n,\ell}\in\N$ such that
\begin{align*}
&A^{(n)}_{\ell,\gamma_n} = \bigcup_{j=1}^{L_{n,\ell}}A^{(n)}_{\ell,\gamma_n}(j), &\vol(A^{(n)}_{\ell,\gamma_n}(j)\cap A^{(n)}_{\ell,\gamma_n}(j')) = 0, \ \ \ \forall j\neq j'.
\end{align*}
For each $A^{(n)}_{\ell,\gamma_n}(j)$, denote by $\lambda^{(n), \max}_{\ell,\gamma_n}(j)$ and $\lambda^{(n), \min}_{\ell,\gamma_n}(j)$
the radius of the smallest circumscribed ball and the largest inscribed ball in $A^{(n)}_{\ell,\gamma_n}(j)$, respectively.
Define 
\begin{align}
&\Lambda_n:= \max_{\substack{1\leq j\leq L_{n,\ell}\\ 1\leq \ell\leq T_{n,\gamma_n}}}\frac{\delta_n}{\lambda^{(n), \min}_{\ell,\gamma_n}(j)}&r_n:= \max_{\substack{1\leq j\leq L_{n,\ell}\\ 1\leq \ell\leq T_{n,\gamma_n}}} \frac{\lambda^{(n),\max}_{\ell,\gamma_n}(j)}{\lambda^{(n), \min}_{\ell,\gamma_n}(j)},
\end{align}
where $\delta_n$ is defined in \eqref{mydelta}. If
\begin{align}
&\limsup_{n\to\infty}\frac{\log \max\{r_n, \Lambda_n\}}{\log N_n}<\infty,\label{lalalalala}
\end{align}
then the second part of assumption 1 and assumption 3 hold in Theorem \ref{thm:new} with $\kappa$ bounded by \eqref{lalalalala}. 
\end{lemma}

$\Lambda_n$ and $r_n$ approximately measure the size ratio of $f$-weighted $\epsilon_n$-coverings relative to the partitioning sets $A^{(n)}_{\ell,\gamma_n}(j)$ and the degeneracy of $A^{(n)}_{\ell,\gamma_n}(j)$, respectively. 

The following corollary is an immediate consequence of Theorem \ref{thm:new} and Lemma \ref{mylemma}:

\begin{corollary}\label{cor:pfinite}
Assume that $D$ is a bounded set with non-empty interior, and $V_n$ contains continuously differentiable functions for every $n$. If
\begin{align*}
\beta: = \limsup_{n\to\infty}\frac{\log\left(\frac{\max_{x\in D}R_n(x)}{\min_{x\in D}R_n(x)}\right)}{\log N_n}<\infty, 
\end{align*}
then
\begin{align}\label{eq:pbound}
  p^*&\leq d\beta + \limsup_{n \rightarrow \infty} \frac{\log\left(\int_D R_n(x)^d\dx{x}\right)}{\log N_n}.
\end{align}
\end{corollary}

\begin{proof}
As discussed earlier, taking $\gamma_n = \frac{\delta_n}{\epsilon_n}$, the first part of the assumption 1 and assumption 2 in Theorem \ref{thm:new} hold simultaneously with $\zeta = \beta$ and $q=0$.
Thus, it remains to verify the second part of the assumption 1 and assumption 3, for which we apply Lemma \ref{mylemma}. 
Based on our choice of $\gamma_n$, $T_{n,\gamma_n}=1$ for all $n\in\N$, i.e., $A^{(n)}_{\ell,\gamma_n} = D$, which is independent of $n$ and $\ell$. 
Now take the partition of $A^{(n)}_{\ell,\gamma_n}$ as itself.
Since $D$ is bounded and has non-empty interior, it contains a small ball with radius $r>0$, and is contained in a large ball with radius $r'>0$.  
Thus, $r_n$ defined in \eqref{gdd2} satisfies $r_n\leq\frac{r'}{r}$ for all $n$. 
On the other hand, as $\{V_n\}_{n=0}^\infty$ is hierarchical, $\delta_n$ is a non-increasing sequence in $n$, i.e., $\Lambda_n\leq\delta_n/r\leq\delta_1/r$. According to Lemma \ref{mylemma}, the assumption 3 is satisfied with $\kappa=0$. The proof is finished by appealing to Theorem \ref{thm:new}.
\end{proof}

\begin{remark}
This result is equivalent to uniformly covering $D$ with $\epsilon_n$-balls, although such a choice is often not optimal in terms of estimating the true value of $p^*$.  
More refined estimates of $p^*$ can be obtained by applying the strategy in Theorem \ref{thm:new}, which unfortunately are not computationally feasible in practice. 
\end{remark}

We provide the proofs for Lemma \ref{mylemma} and Theorem \ref{thm:new} in the next section, and end our section here by providing examples demonstrating finiteness of $p^\ast$, all of which satisfy the assumptions of Corollary \ref{cor:pfinite}, and hence leverage \eqref{eq:pbound}.

\begin{example}[Complex exponentials]
  Consider the setup of Example \ref{ex:complex-exp}, and recall that $F_n \in \Z^d$ denotes the set of frequencies associated to subspace $V_n$. Direct computation with the basis \eqref{eq:vn-complex-exp} shows that $\beta = 0$, and that
  \begin{align*}
    R_n(x)^2 &= \sum_{j=1}^{N_n} \|f_j\|^2 \leq N_n \max_{j \in [N_n]} \|f_j\|^2, & F_n &= \left\{ f_1, \ldots, f_{N_n}\right\},
  \end{align*}
  where $\|\cdot\|$ is the standard $\ell^2$ norm on elements of $\Z^d$. Therefore, 
  \begin{align*}
    p^\ast \leq \frac{d}{2} \left[ 1 + \lim_{n \rightarrow \infty} \frac{\log \max_{f \in F_n} \|f\|^2}{\log N_n} \right]
  \end{align*}
  Thus, provided that the maximum frequency grows at most exponentially with $N_n$, then $p^\ast < \infty$, but does depend on the dimension $d$. In particular, if $F_n = \left\{-n, ..., n\right\}^d$, then $p^\ast \leq \frac{d}{2} + 1$.
\end{example}

\begin{example}[Tensor-product Chebyshev polynomials]\label{ex:chebp}
  Consider the setup of Example \ref{ex:cheb}, with $V_n = P_n$; we show in this example that $p^\ast \leq 4 d$. First we note that since $P_n$ contains linear functions for $n \geq 1$, then there is a universal constant $C > 0$ so that
  \begin{align*}
    R_n(x) \geq R_1(x) \geq C \hskip 10pt \Longrightarrow \hskip 10pt 
  \beta \leq \limsup_{n \rightarrow \infty} \frac{\log\max_{x \in D} R_n(x)}{\log N_n}.
  \end{align*}
  To compute the maximum of $R_n$, we use the univariate orthonormal Chebyshev polynomials $\{T_j\}_{j \in \N}$ to form the multivariate orthonormal basis,
  \begin{align*}
    T_\lambda(x) &= \prod_{j=1}^d T_{\lambda^{(j)}}(x_j), & \lambda = (\lambda^{(1)}, \ldots, \lambda^{(d)}) \in \N^d,
  \end{align*}
  and for fixed $n$ assign $v_i(x) = T_{\lambda_i(x)}$ for $\{\lambda_i\}_{i \in N_n}$ any enumeration of the multi-index set $\Sigma_n$ in \eqref{eq:Sigman}. By using the fact that,
  \begin{align*}
    \ddx{x} T_k(x) = \sqrt{2} k U_{k-1}(x),
  \end{align*}
  with $U_k(x)$ the degree-$k$ univariate Chebyshev polynomial of the second kind, then direct computation shows that,
  \begin{align*}
    \max_{x \in D} \left\| \nabla v_j(x) \right\|^2 \leq 2 \left(\max_{q} \lambda^{(j)}_q\right)^2 \sum_{q=1}^d \max_{x \in D} U^2_{\lambda^{(j)}-e_q}(x),
  \end{align*}
  where $e_q \in \N^d$ is the cardinal unit vector in direction $q$.
  We use \cite[Theorem 8]{migliorati_multivariate_2015} to bound $U^2_{\lambda^{(j)}-e_q}(x)$,
  \begin{align*}
    \max_{x \in D} U^2_{\lambda^{(j)}-e_q}(x) \leq N_{n-1}^3,
  \end{align*}
  which holds since $\left|\lambda^{(j)}-e_q\right| \leq n-1$. Since $\lambda^{(j)}_q \leq n$, we conclude for any $t \geq 1$,
  \begin{align*}
    \max_{x \in D} \left\| \nabla v_j(x) \right\|^2 \leq 2 n^2 d N_{n-1}^3 \hskip 10pt \Longrightarrow \hskip 10pt \max_{x \in d} R_n(x)^t \leq \left( 2 n^2 d N_n^4 \right)^{t/2}.
  \end{align*}
  Thus, we have
  \begin{align*}
    \beta &\leq \limsup_{n \rightarrow \infty} \frac{\log \max_{x \in D} R_n(x)}{\log N_n} = 2, & 
    \limsup_{n \rightarrow \infty} \frac{\log\left(\int_D R_n(x)^d\dx{x}\right)}{\log N_n} &\leq 2 d,
  \end{align*}
  and thus by \eqref{eq:pbound} we have $p^\ast \leq 4 d$.
\end{example}
We emphasize that the computations above are not sharp estimates of $p^\ast$. Corollary \ref{cor:pfinite} itself is loose since it essentially employs a uniform $\epsilon_n$ covering instead of a more refined $R_n^{-1}$-weighted covering. In addition, many of the computations in, e.g., Example \ref{ex:chebp}, employ crude estimates. 

A strategy nearly identical to that presented in Example \ref{ex:chebp} can be used to show finite $p^\ast$ for tensorized Jacobi polynomials introduced in Example \ref{ex:jacobi}. For brevity we omit this computation.

\subsection{Proofs of Lemma \ref{mylemma} and Theorem \ref{thm:new}}

\begin{proof}[Proof of Lemma \ref{mylemma}]
First note that as $R_n$ is continuous, $A^{(n)}_{\ell,\gamma_n}(j)$ has non-empty interior, so that $r^{(n)}_{\ell,\gamma_n}(j)$ is finite, i.e., $r_n<\infty$ for every $n\in\N$. 

Now fix $n, \ell, j$. 
By definition of $r_n$, $A^{(n)}_{\ell,\gamma_n}(j)$ contains a ball with radius ${\lambda^{(n), \min}_{\ell,\gamma_n}(j)}$, meanwhile is contained in a ball with radius $r_n{\lambda^{(n), \min}_{\ell,\gamma_n}(j)}$. 
Without loss of generality, we assume the latter is centered at origin, i.e., $A^{(n)}_{\ell,\gamma_n}(j)\subset B_{r_n{\lambda^{(n), \min}_{\ell,\gamma_n}(j)}}(0)$.  
In this case, 
\begin{align*}
\vol(A^{(n)}_{\ell,\gamma_n}(j))\geq \vol(B_{\lambda^{(n), \min}_{\ell,\gamma_n}(j)}(0)).
\end{align*}
On the other hand, a minimal $\epsilon$-covering of $B_{r_n{\lambda^{(n), \min}_{\ell,\gamma_n}(j)}}(0)$ is an $\epsilon$-exterior-covering of $A^{(n)}_{\ell,\gamma_n}(j)$, where an exterior covering allows covering points to be outside the set.
Denote by $\mathsf N^{\text{ext}}(\cdot, \epsilon)$ the standard $\epsilon$-exterior-covering number of a set.
Utilizing a relationship between exterior covering numbers and interior covering numbers \cite[Exercise 4.2.9]{vershynin2018high}, we have 
\begin{align}
\mathsf N(A^{(n)}_{\ell,\gamma_n}(j), \gamma_n^{\ell-q-1}\epsilon_n)\leq \mathsf N^{\text{ext}}(A^{(n)}_{\ell,\gamma_n}(j), \frac{1}{2}\gamma_n^{\ell-q-1}\epsilon_n)&\leq \mathsf N(B_{r_n{\lambda^{(n), \min}_{\ell,\gamma_n}(j)}}(0), \frac{1}{2}\gamma_n^{\ell-q-1}\epsilon_n)\nonumber\\
&\leq \left(1+\frac{4r_n{\lambda^{(n), \min}_{\ell,\gamma_n}(j)}}{\gamma_n^{\ell-q-1}\epsilon_n}\right)^d\nonumber\\
&\leq  \left[\left(4r_n+\frac{\Lambda_n}{4}\right)\frac{{\lambda^{(n), \min}_{\ell,\gamma_n}(j)}}{\gamma_n^{\ell-q-1}\epsilon_n}\right]^d\label{gdd2},
\end{align}
where the penultimate inequality follows from the fact that $\mathsf N(B_r(0), \epsilon)\leq (1+2r/\epsilon)^d$ \cite[Corollary 4.2.13]{vershynin2018high}, and the last inequality follows from
\begin{align*}
\frac{{\lambda^{(n), \min}_{\ell,\gamma_n}(j)}}{\gamma_n^{\ell-q-1}\epsilon_n}\geq\frac{{\lambda^{(n), \min}_{\ell,\gamma_n}(j)}}{\delta_n}\geq \frac{1}{\Lambda_n}.
\end{align*}
Consequently, 
\begin{align*}
\frac{\mathsf N(A^{(n)}_{\ell,\gamma_n}(j), \gamma_n^{\ell-q-1}\epsilon_n)\vol(B_{\gamma_n^{\ell-q-1}\epsilon_n}(0))}{\vol(A^{(n)}_{\ell,\gamma_n}(j))}&\stackrel{\eqref{gdd1}, \eqref{gdd2}}{\leq} \frac{\left[\left(4r_n+\frac{\Lambda_n}{4}\right)\frac{{\lambda^{(n), \min}_{\ell,\gamma_n}(j)}}{\gamma_n^{\ell-q-1}\epsilon_n}\right]^d\vol(B_{\gamma_n^{\ell-q-1}\epsilon_n}(0))}{\vol(B_{\lambda^{(n), \min}_{\ell,\gamma_n}(j)}(0))}\\
&\leq \left(4r_n+\frac{\Lambda_n}{4}\right)^d. 
\end{align*}
Hence,
\begin{align}
\frac{\mathsf N(A^{(n)}_{\ell,\gamma_n}, \gamma_n^{\ell-q-1}\epsilon_n)\vol(B_{\gamma_n^{\ell-q-1}\epsilon_n}(0))}{\vol(A^{(n)}_{\ell,\gamma_n})}&\leq \frac{\sum_{j=1}^{L_{n,\ell}}\mathsf N(A^{(n)}_{\ell,\gamma_n}(j), \gamma_n^{\ell-q-1}\epsilon_n)\vol(B_{\gamma_n^{\ell-q-1}\epsilon_n}(0))}{\sum_{j=1}^{L_{n,\ell}}\vol(A^{(n)}_{\ell,\gamma_n}(j))}\nonumber\\
&\leq\max_{j\in [L_{n,\ell}]}\left\{\frac{\mathsf N(A^{(n)}_{\ell,\gamma_n}(j), \gamma_n^{\ell-q-1}\epsilon_n)\vol(B_{\gamma_n^{\ell-q-1}\epsilon_n}(0))}{\vol(A^{(n)}_{\ell,\gamma_n}(j))}\right\}\nonumber\\
&\leq \left(4r_n+\frac{\Lambda_n}{4}\right)^d. \label{gdd3}
\end{align}
Since the right-hand side of \eqref{gdd3} is independent of $\ell$, $w_n$ in the assumption 3 in Theorem \ref{thm:new} can be taken as $w_n = (4r_n+\frac{\Lambda_n}{4})^d$, which combined with condition \eqref{lalalalala} finishes the proof. 
\end{proof}

\begin{proof}[Proof of Theorem \ref{thm:new}]
An upper bound for $p^*$ can be obtained by considering an $f$-weighted $\epsilon_n$-covering in $\mathcal S_{f, \epsilon_n}(D)$. 
We construct such a covering by combining a sequence of coverings on the partition sets $A^{(n)}_{\ell,\gamma_n}$. 

Note by the definition of $A^{(n)}_{\ell,\gamma_n}$, the local radius $F_0(x)$ of any $x\in A^{(n)}_{\ell,\gamma_n}$ satisfies
\begin{align*}
F_0(x) \leq\gamma_n^{\ell}\epsilon_n.
\end{align*}
By \eqref{nice}, the $f$-weighted $\epsilon_n$-covering radius of $x\in A^{(n)}_{\ell,\gamma_n}$ satisfies
\begin{align}
r(x, \epsilon_n)=\epsilon_n F_{r(y, \epsilon_n)}(x)\geq\epsilon_n F_{0}(x)\geq\epsilon_n F_{\gamma_n^{\ell}\epsilon_n}(x).\label{myneed1}
\end{align}
According to assumption 2, 
\begin{align*}
F_{\gamma_n^{\ell}\epsilon_n}(x) = \frac{\min_{z\in B_{\gamma_n^{\ell}\epsilon_n}(x)\cap D}f(z)}{m_f}\geq\frac{\min_{z\in \cup_{|\ell'-\ell|\leq q}A^{(n)}_{\ell',\gamma_n}} f(z)}{m_f}\geq\frac{\gamma_n^{\ell-q-1}m_f}{m_f} = \gamma_n^{\ell-q-1},
\end{align*}
which can be plugged into \eqref{myneed1} to yield
\begin{align*}
&r(x, \epsilon_n)\geq \gamma_n^{\ell-q-1}\epsilon_n&\forall x\in A^{(n)}_{\ell,\gamma_n}. 
\end{align*}
Thus, the $f$-weighted $\epsilon_n$-covering number of $A^{(n)}_{\ell,\gamma_n}$ is bounded by $\mathsf N(A^{(n)}_{\ell,\gamma_n}, \gamma_n^{\ell-q-1}\epsilon_n)$. 
Now let $\mathcal N_\ell^{(n)}\subset A^{(n)}_{\ell,\gamma_n}$ be an optimal $f$-weighted $\epsilon_n$-covering for $A^{(n)}_{\ell,\gamma_n}$. Define
\begin{align}
\mathcal N = \bigcup_{\ell\in [T_{n,\gamma_n}]}\mathcal N_\ell^{(n)},\label{gdd1}
\end{align} 
which is an $f$-weighted $\epsilon_n$-covering for $D$. 
Under such a choice of $\mathcal N$, it is easy to verify
\begin{align*}
G_\mathcal N = \max_{y\in\mathcal N}\left(\frac{\max_{x\in B_{r(y, \epsilon)}\cap D}f(x)}{\min_{x\in B_{r(y, \epsilon)}\cap D}f(x)}\right)^d &= \max_{\ell\in [T_{n,\gamma_n}]}\max_{y\in \mathcal N_\ell^{(n)}}\left(\frac{\max_{x\in B_{r(y, \epsilon)}\cap D}f(x)}{\min_{x\in B_{r(y, \epsilon)}\cap D}f(x)}\right)^d\\
&\leq \max_{\ell\in [T_{n,\gamma_n}]}\left(\frac{\max_{z\in \cup_{|\ell'-\ell|\leq q}A^{(n)}_{\ell',\gamma_n}} f(z)}{\min_{z\in \cup_{|\ell'-\ell|\leq q}A^{(n)}_{\ell',\gamma_n}} f(z)}\right)^d\\
&\leq \gamma_n^{2(q+1)d}.
\end{align*}
Moreover, the cardinality of $\mathcal N$ can be bounded using assumption 3 as follows:
\begin{align*}
|\mathcal N| = \sum_{\ell\in [T_{n,\gamma_n}]}|\mathcal N_\ell^{(n)}| &\leq \sum_{\ell\in [T_{n,\gamma_n}]}\mathsf N(A^{(n)}_{\ell,\gamma_n}, \gamma_n^{\ell-q-1}\epsilon_n)\\
&\leq w_n\sum_{\ell\in [T_{n,\gamma_n}]}\frac{\vol(A^{(n)}_{\ell,\gamma_n})}{\vol(B_{\gamma_n^{\ell-q-1}\epsilon_n}(0))}\\
& = w_n\vol(B_1(0))\sum_{\ell\in [T_{n,\gamma_n}]}\left(\gamma_n^{-\ell+q+1}\epsilon^{-1}_n\right)^d\vol(A^{(n)}_{\ell,\gamma_n})\\
& = w_n\vol(B_1(0))(3\gamma^{q+1}_n)^d\sum_{\ell\in [T_{n,\gamma_n}]}  \left(\gamma_n^{-\ell}\max_{x\in D}R_n(x)\right)^d \vol(A^{(n)}_{\ell,\gamma_n})\\
&\leq w_n\vol(B_1(0))(3\gamma^{q+1}_n)^d\int_D R_n(x)^d\dx{x}.
\end{align*}
Thus, 
\begin{align*}
p^*&\leq\limsup_{n\to\infty}\frac{\log\left(|\mathcal N|G_\mathcal N\right)}{\log N_n}\\
&\leq \limsup_{n\to\infty}\frac{\log w_n}{\log N_n}+\limsup_{n\to\infty}\frac{(q+1)d\log\gamma_n}{\log N_n}+\limsup_{n\to\infty}\frac{\log\left(\int_D R_n(x)^d\dx{x}\right)}{\log N_n}\\
&\leq \kappa + (q+1)d\zeta+\limsup_{n\to\infty}\frac{\log\left(\int_D R_n(x)^d\dx{x}\right)}{\log N_n},
\end{align*}
where the last inequality follows from assumption 1, completing the proof.
\end{proof}

\section{Randomized admissible meshes}\label{4}
The previous section focuses on devising sampling strategies to obtain an optimal sampling size at the expense of rendering the equivalence coefficient $C_n$ being reasonably large. 
Alternatively, one may expect $C_n=\mathcal O(1)$ when the sampling size is sufficient, so that randomly generated grids almost fill the domain. This is similar to the idea in \cite{calvi_uniform_2008} but avoids deterministically discretizing domains with complicated geometry. With random samples, such a result is not possible using the WAM analysis in Section \ref{3}. 
One expects that attaining an admissible mesh from random samples is asymptotically possible (i.e., as the sample count increases to infinity) for all sampling measures that are absolutely continuous with respect to the uniform measure on $D$. However, in terms of finite-sample behavior, different measures may demonstrate drastically unequal performance; our analysis quantifying this relies on the weighted covering machinery introduced in the previous section. We will first consider taking the uniform measure as the baseline, after which we design a sampling measure that achieves better efficiency.
We recall that when discussing AM's, we assume the domain $D$ is both convex and compact. In particular, $D$ has a non-empty interior.  
The analysis of the sampling strategies in Section \ref{611} and \ref{612} relies on a few technical lemmas that develop next.

\subsection{Uniform interior cone condition and volume ratios}

We introduce the following variant of the uniform interior cone condition that is used in \cite{Bramson_2013}:

\begin{definition}[Weak uniform interior cone condition]\label{def:uicc}
A domain $D$ is said to satisfy the weak uniform interior cone condition with radius parameter $\delta>0$ and angle parameter $\theta\in (0,\frac{\pi}{2}]$ if for every $x\in D$, there exists at least one unit vector $v_x\in\R^d$ such that the cone $C(v_x):=\{z: \langle z, v_x\rangle>|z|\cos\theta\}$ satisfies
\begin{align}
&(x+C(v_x))\cap B_\delta(x)\subseteq D,\label{mycone}
\end{align}
where $x+C(v_x)$ is the Minkowski sum of $\{x\}$ and $C(v_x)$. 
\end{definition}

We show next that the weak uniform interior cone condition in Definition \ref{def:uicc} implies the following uniform volume ratio condition, which is also known as the measure density condition in the literature \cite{Haj_asz_2008}.

\begin{lemma}\label{lem:mdc1}
If $D\subset\R^d$ is a compact domain that satisfies the weak uniform interior cone condition with radius parameter $\delta$ and angle parameter $\theta$, then, for any $0<r_0<\infty$, there exists a constant $C$, which depends only on $D$ and $r_0$, such that 
 \begin{align}
&\vol (B_r(x)\cap D)\geq C\cdot\vol(B_r(x)),&\forall x\in D, 0<r\leq r_0.\label{tech}
\end{align}
\end{lemma}

\begin{proof}
It suffices to show 
 \begin{align}
&\vol (B_r(x)\cap D)\geq C'\cdot\vol(B_r(x)),&\forall x\in D, 0<r\leq h\label{smalltech}
\end{align}
holds for some small constant $0<h\leq r_0$ and finite constant $C'$. Indeed, if \eqref{smalltech} holds, then for any $h<r\leq r_0$ and $x\in D$, 
\begin{align}
\vol (B_r(x)\cap D)\geq\vol (B_h(x)\cap D)\geq C'\cdot\vol(B_h(x)) \geq C'\left(\frac{r_0}{h}\right)^d\cdot\vol(B_r(x)).
\end{align}
Setting $C = C'(\frac{r_0}{h})^d$ yields the desired result in \eqref{tech}. 

To show \eqref{smalltech} holds for sufficiently small $h$, since $D$ satisfies the weak uniform interior cone condition with radius parameter $\delta$ and angle parameter $\theta$, for any $0<r\leq\delta$ and $x\in D$,
\begin{align}
\vol (B_r(x)\cap D)\stackrel{\eqref{mycone}}{\geq}\vol ((x+C(v_x))\cap B_\delta(x))&\geq\frac{\delta\cos\theta}{d}\cdot(\delta\sin\theta)^{d-1}\omega_{d-1}\nonumber\\
& = \underbrace{\frac{\cos\theta}{d}\cdot(\sin\theta)^{d-1}\frac{\omega_{d-1}}{\omega_d}}_{: = \widetilde{C}}\vol(B_r(x)),\label{pou}
\end{align}
where $w_s$ is the volume of the unit ball in $\R^{s}$. 
Note $\widetilde{C}$ defined in \eqref{pou} is positive and independent of $x$ and $r$. Thus, \eqref{smalltech} holds with $h = \delta$ and the same $C'=\widetilde{C}$.
\end{proof}

The weak uniform interior cone condition is reasonable under mild regularity assumptions on the domain.
In particular, convex compact sets in $\R^d$ with non-empty interior satisfy the weak interior cone condition: 

\begin{lemma}\label{lem:mdc2}
If $D\subset\R^d$ is a compact convex set with non-empty interior, then $D$ satisfies the weak uniform interior cone condition.
\end{lemma}

\begin{proof}
As $D$ is compact, $D$ is bounded and denote its diameter as $L<\infty$. 
Since $D$ has a non-empty interior, there exists $x\in D$ and $\delta>0$ such that $B_{3\delta}(x)\subset D$. 
It is clear that for all $z\in B_{2\delta}(x)$ with $r\leq \delta$, $B_r(z)\subset B_{3\delta}(x)\subset D$, i.e., \eqref{mycone} is satisfied for $z\in B_{2\delta}(x)$ with radius parameter $\delta$ and any angle parameter $\theta\in (0,\frac{\pi}{2}]$.  

On the other hand, for any $z\in D\setminus B_{2\delta}(x)$, since $D$ is convex, the convex hull of $B_\delta(x)$ and $z$, denoted by $CH_z$, is contained in $D$. Choose $v_z = \frac{x-z}{\|x-z\|}$. 
Since the distance between $z$ and $B_\delta(x)$ is at least $\delta$, it can be verified that for any $0<r\leq \delta$, 
\begin{align*}
&(z+\{y: \langle y, v_z\rangle>|y|\cos\nu\})\cap B_r(z)\subset CH_z\subset D&\nu = \arcsin\left(\frac{\delta}{L}\right).
\end{align*}
Putting both cases together, we conclude that $D$ satisfies the weak uniform interior cone condition with radius parameter $\delta$ and angle parameter $\nu$. 
\end{proof}

Lemma \ref{lem:mdc1} and Lemma \ref{lem:mdc2} together implies the following theorem:

\begin{theorem}
If $D\subset\R^d$ is a compact convex domain, for any $r_0>0$, there exists a constant $C$ that depends only on $r_0$ and $D$, such that
\begin{align}
&\vol (B_r(x)\cap D)\geq C\cdot\vol(B_r(x)),&\forall x\in D, 0<r\leq r_0.\label{mygoal}
\end{align}
\end{theorem}

In the rest of the section, we will denote by $\alpha_{D,r_0}$ the largest constant such that \eqref{mygoal} is satisfied, i.e.,
\begin{align}
\alpha_{D,r_0} = \sup\{C: \text{\eqref{mygoal} is satisfied}\}.\label{myalpha}
\end{align}

\subsection{Sampling from the uniform measure}\label{611}
Suppose that $\mathcal A_n$ is a set of points independently and uniformly sampled from $D$. 
The next theorem shows that when $M_n$ is logarithmically larger than the cardinality of some optimal covering of $D$, with overwhelming probability, $b=0$. 
\begin{theorem}\label{th:am}
  Given $\mu$, $D$, and $\left\{ V_n \right\}_{n=0}^\infty$, define
  \begin{align*}
    &\calA_n \coloneqq \left\{X_{m} \right\}_{m=1}^{M_n}& X_m \stackrel{\text{i.i.d.}}{\sim}\mathsf{Unif}(D).
  \end{align*}
  Assume that $D$ is compact and convex and elements in $V_n$ are twice continuously differentiable. 
  Fix $k>2$ and $r> 1$. Let $\mathcal N_n\subset D$ be a $(kR_\mu(V_n))^{-1}$-covering of $D$ with $|\mathcal N_n|\geq N_n$. 
  There exists a constant $C>0$ that does not depend on $n$, such that if $M_n \geq C(r+1)|\mathcal N_n|\log |\mathcal N_n|$, then with probability at least $1-N_n^{-r}$,  
  \begin{align*}
  \|v\|_\infty\leq\frac{k}{k-2}\|v\|_{\mathcal A_n, \infty}.
  \end{align*}
\end{theorem}
\begin{proof}
Take 
$$v(x)=\sum_{i\in [N_n]}\langle v, v_i\rangle_\mu v_i(x)\in V_n.$$
Without loss of generality assume $\|v\|_{2,\mu}=\sum_{i\in [N_n]}\langle v, v_i\rangle_\mu^2=1$, otherwise consider $v/\|v\|_{2, \mu}$. 
For $x,y\in D$, the segment connecting $x$ and $y$ is in $D$ under the convexity assumption. 
Set $L = \|y-x\|_2$ and $z = L^{-1}(y-x)$. 
Applying the Fundamental Theorem of Calculus, we have
\begin{align}
\left|v(y)-v(x)\right| &= \left|\int_0^{L}\langle\nabla v(x+tz),z\rangle \dx{t}\right| \\
&= \left|\sum_{i\in [N_n]} \langle v, v_i\rangle_{\mu}\int_0^{L}\langle\nabla v_i(x+tz), z\rangle \dx{t}\right|\nonumber\\
&\leq \left(\sum_{i\in [N_n]}\left|\int_0^{L}\langle\nabla v_i(x+tz), z\rangle \dx{t}\right|^2\right)^{1/2}\nonumber\\
&\leq \left(\sum_{i\in [N_n]}L\int_0^{L}\left|\langle\nabla v_i(x+tz), z\rangle\right|^2 \dx{t}\right)^{1/2}\nonumber\\
&\leq\left(L\int_0^{L}\sum_{i\in [N_n]}\|\nabla v_i(x+tz)\|^2 \ \dx{t}\right)^{1/2}\leq LR_\mu(V_n),\label{am:1} 
\end{align}
which implies that $|v(x)-v(y)|<\epsilon$ if $\|x-y\|_2<R_\mu(V_n)^{-1}\epsilon$.   Now take $\epsilon = 1/k$, and let 
\begin{align}
\eta = (kR_\mu(V_n))^{-1}\leq (kR_\mu(V_0))^{-1}\leq r_0: = (k\min_{x\in D}R_0(x))^{-1},\label{eta}
\end{align} 
where the inequality is due to the hierarchical structure of $V_n$.

Let $\mathcal N_n\subset D$ be an $\eta$-covering of $D$ with $|\mathcal N_n|\geq N_n$. 
For any $y\in\mathcal N_n$, 
\begin{align}
\vol(D)\leq\vol\left(\cup_{z\in\mathcal N_n}B_\eta(z)\right) = |\mathcal N_n|\cdot\vol(B_\eta(y))\Longrightarrow\vol(B_\eta(y))\geq\frac{\vol(D)}{|\mathcal N_n|}.\label{tirdy}
\end{align}
Suppose that $\mathcal A_n\subset D$ satisfies
\begin{align}
&\mathcal A_n\cap B_{\eta}(y)\neq\varnothing&\forall y\in\mathcal A_n. \label{event}
\end{align}  
It follows from the triangle inequality that for every element in $D$ one can find a point in $\mathcal A_n$ such that their distance is at most $2\eta$. Therefore, 
\begin{align}
\frac{\|v\|_\infty}{\|v\|_{\mathcal A_n, \infty}}\stackrel{k>2}{\leq}\frac{\|v\|_\infty}{\|v\|_{\infty}-\frac{2}{k}}\stackrel{\|v\|_\infty\geq \|v\|_{2,\mu}=1}{\leq}\frac{k}{k-2}.\label{am:am}
\end{align}
We now show that for sufficiently large $M_n$, with high probability $\mathcal A_n$ satisfies \eqref{event}.
In fact, for $y\in\mathcal N_n$, 
\begin{align*}
\mathrm{Pr}\left[B_ {\eta}(y)\cap\{X_m\}_{m\in [M_n]}=\varnothing\right] &= \mathrm{Pr}\left[B_ {\eta}(y)\cap X_1=\varnothing\right]^{M_n}\\
&=\left(1-\frac{\vol(B_{\eta}(y)\cap D)}{\vol(D)}\right)^{M_n}\\
&\stackrel{\eqref{myalpha}}{\leq}\left(1-\frac{\alpha_{D,r_0}\vol(B_{\eta}(y))}{\vol(D)}\right)^{M_n}\\
&\stackrel{\eqref{tirdy}}{\leq}\left(1-\frac{\alpha_{D,r_0}}{|\mathcal N_n|}\right)^{M_n}
\leq e^{-\frac{\alpha_{D,r_0}M_n}{|\mathcal N_n|}}.
\end{align*}
Taking a union bound over $y$ yields that the probability of $\mathcal A_n$ not satisfying \eqref{event} is at most $|\mathcal N_n|e^{-\frac{\alpha_{D,r_0}M_n}{ |\mathcal N_n|}}$. Setting $M_n = \alpha^{-1}_{D,r_0}(r+1)|\mathcal N_n|\log|\mathcal N_n|$ completes the proof. 
\end{proof}

\subsection{Sampling from $\nu_n$}\label{612}
Note that in the derivation of \eqref{am:1} the integrand is directly bounded by $R_\mu(V_n)$, which is a global quantity of $V_n$. 
We now propose an alternative random sampling strategy which aims to exploit the local structures of $V_n$ so that the analysis of covering becomes more efficient. 
Precisely, we wish to cover points with large gradients using smaller balls.  
Keeping the steps before the last step in \eqref{am:1} yields
\begin{align}
|v(y)-v(x)|\leq L\underbrace{\left(\frac{1}{L}\int_0^L R_n(x+tz)^2\dx{t}\right)^{1/2}}_{(*)}.\label{am:loc}
\end{align}
When $L$ is small, $(*)\approx R_n(x)$, so the local Lipschitz constant of $v$ at $x$ is bounded by $R_n(x)$. 
In other words, evaluation of $v$ at points within $R_n(x)^{-1}\epsilon$ distance to $x$ change from $v(x)$ by approximately $\epsilon$. 
Based on this observation, we propose an alternative sampling strategy making use of such local information, resulting in the Theorem below.

\begin{theorem}\label{th:amn}
Under the same assumption in Theorem \ref{th:am}, let
\begin{align*}
X_m\stackrel{\text{i.i.d.}}{\sim}\nu_n,
\end{align*}
where $\nu_n$ is defined in \eqref{am:nu}.
Fix $k>2$ and $r>1$. 
Let $\mathcal N_n\subset D$ be an $R(x)^{-1}$-weighted $(kR_\mu(V_n))^{-1}$-covering of $D$. 
There exists a constant $C>0$ that does not depend on $n$, such that if $M_n \geq C(r+1)G_{\mathcal{N}_n}|\mathcal N_n|\log |\mathcal N_n|$ with $|\mathcal N_n|\geq N_n$, then with probability at least $1-N_n^{-r}$, 
  \begin{align*}
  \|v\|_\infty\leq\frac{k}{k-2}\|v\|_{\mathcal A_n, \infty}.
  \end{align*}
\end{theorem}

Before giving the proof, we note that Theorem \ref{th:amn} combined with definition \eqref{p^*} immediately implies the following result:

\begin{theorem}\label{thm:nun-AM}
Under the same condition as in Theorem \ref{th:amn} and assuming $p^*$ defined in \eqref{p^*} is finite, with probability $1$ (where the underlying probability space is the product space of the samples $X_m$), the sequence $\{\mathcal A_n\}_{n=0}^\infty$ generated by $\nu_n$ forms an asymptotic AM for $\{V_n\}_{n=0}^\infty$ with $a = p^{*}+\tau$ and $b=0$ for any $\tau>0$. 
\end{theorem}

\begin{proof}
Since $p^*$ defined in \eqref{p^*} is finite, for any $\tau>0$, there exists a sequence $\{\mathcal N_n\}_{n=0}^\infty$ such that $\mathcal N_n$ is an $f$-weighted $(3R_\mu(V_n))^{-1}$-covering of $D$ and
\begin{align*}
\limsup_{n\to\infty}\frac{\log\left(|\mathcal {N}_n| G_{\mathcal {N}_n}\right)}{\log N_n} <p^*+\frac{\tau}{2}.
\end{align*} 
The corresponding $M_n$ derived in Theorem \ref{th:amn} satisfies
\begin{align*}
\limsup_{n\to\infty}\frac{M_n}{\log N_n}&=\limsup_{n\to\infty}\left(\frac{\log C}{\log N_n}+\frac{\log(r+1)}{\log N_n}+\frac{\log\left(|\mathcal {N}_n| G_{\mathcal {N}_n}\right)}{\log N_n}+\frac{\log\log |\mathcal N_n|}{\log N_n}\right)\\
& <p^*+\frac{\tau}{2}.
\end{align*}
Thus, there exists some constant $C_\tau$ such that $M_n\leq C_\tau N_n^{p^*+\tau}$ for all $n\in\N$. 
On the other hand, Theorem \ref{th:amn} tells us that for every $n$ with sampling size $M_n$,  
\begin{align*}
\text{Pr}\left[ \|v\|_\infty> 3\|v\|_{\mathcal A, \infty}\right]\leq N_n^{-r}\leq n^{-r}. 
\end{align*} 
The proof is completed by an application of the Borel-Cantelli lemma as in the proof of Theorem \ref{thm:mu-wam}.  
\end{proof}

Since $\nu_{n} \ll \mu$ for any $n$, then combining the above result with Theorem \ref{thm:WAM-con}, produces a (classical) AM.
\begin{corollary}\label{cor:nun-AM}
  Assume the conditions of Thereom \ref{thm:nun-AM}. If in addition $\{V_n\}_{n=0}^\infty$ is a $\mu$ZC sequence, then with probability 1 $\{\mathcal{A}_n\}_{n=0}^\infty$ is a (classical) AM with exponents $a = p^* + \tau$, and $b = 0$ for any $\tau > 0$.
\end{corollary}

\begin{proof}[Proof of Theorem \ref{th:amn}]
Assume that $\|v\|_{2,\mu}=\sum_{i\in [N_n]}\langle v, v_i\rangle_\mu^2=1$.
Let us consider an $f$-weighted $\eta$-covering of $D$ with 
\begin{align}
f=  R_n(x)^{-1},\label{f}
\end{align}
where $\eta$ is defined in \eqref{eta}.     
Note that every $\eta$-covering of $D$ is also an $f$-weighted $\eta$-covering of $D$; see Section \ref{wcs}.
With some abuse of notation, from now on let $\mathcal N_n$ be an $f$-weighted $\eta$-covering of $D$, i.e., the covering radius $r_y$ for every $y\in\mathcal N_n$ equals $r(y,\eta)\leq r_0$ (see \eqref{eta}), which can always be achieved by elongating $r_y$ if $r_y<r(y,\eta)$. 
A similar computation as before shows that if $\mathcal A_n\subset D$ has non-empty intersection with $B_{r(y, \eta)}(y)$ for every $y\in\mathcal N_n$, then
\begin{align}\label{eq:am-temp}
  \|v\|_\infty\leq \frac{k}{k-2}\|v\|_{\mathcal A_n, \infty}.
\end{align}
To see this, note that the non-empty intersection property implies that for $x\in \argmax_{u\in D} v(u)$, there exist $y\in\mathcal N_n$ and $x'\in\mathcal A_n$ such that $\max\{|y-x|, |y-x'|\}\leq r(y,\eta)$. Therefore,
\begin{align*}
\|v\|_{\mathcal A_n,\infty}\geq v(x')&\geq v(x)-|v(x)-v(y)|-|v(y)-v(x')|\\
&\stackrel{\eqref{am:loc}}{\geq}\|v\|_\infty-2r(y,\eta)\max_{z\in B_{r(y,\eta)}(y)}R_n(z)\\
&\stackrel{\eqref{nice}}{=}\|v\|_\infty-2\eta F_{r(y,\eta)}(y)\max_{z\in B_{r(y,\eta)}(y)}R_n(z)\\
&\stackrel{}{=}\|v\|_\infty-\frac{2}{k}\\
&\stackrel{\|v\|_\infty\geq \|v\|_{2,\mu}=1}{>}\|v\|_\infty -\frac{2\|v\|_\infty}{k},
\end{align*}
which shows \eqref{eq:am-temp}.

We next construct such sets $\mathcal A_n$ satisfying the non-empty intersection property using random sampling.   
To this end, we draw $X_m$ independently from a probability distribution $\nu_n$ which is defined in \eqref{am:nu}.
It is easy to check that for $y\in\mathcal N_n$, 
\begin{align}
\mathrm{Pr}\left[X_m\in B_{r(y,\eta)}(y)\right] = \frac{\int_{B_{r(y,\eta)}(y)\cap D}R_n(x)^d \dx{x}}{\int_D R_n(x)^d\dx{x}}\geq\frac{\int_{B_{r(y,\eta)}(y)\cap D}R_n(x)^d \dx{x}}{\sum_{z\in\mathcal N_n} \int_{B_{r(z,\eta)}(z)\cap D}R_n(x)^d \dx{x}}.\label{as}
\end{align}
The numerator of the last term in \eqref{as} can be bounded from below as
\begin{align}\label{num}
\int_{B_{r(y,\eta)}(y)\cap D}R_n(x)^d \dx{x}&\geq \vol(B_{r(y,\eta)}(y)\cap D)\min_{x\in B_{r(y,\eta)}(y)\cap D}R_n(x)^d\nonumber\\
&\stackrel{\eqref{myalpha}}{\geq} \alpha_{D,r_0}\vol(B_{r(y,\eta)}(y))\min_{x\in B_{r(y,\eta)}(y)\cap D}R_n(x)^d\nonumber\\
&=\alpha_{D,r_0}\vol(B_{1}(0))\cdot r(y,\eta)^d\min_{x\in  B_{r(y,\eta)}(y)\cap D}R_n(x)^d\nonumber\\
&\stackrel{\eqref{nice}}{=} \alpha_{D,r_0}\vol(B_{1}(0))\cdot (\eta F_{r(y,\eta)}(y))^d\min_{x\in B_{r(y,\eta)}(y)\cap D}R_n(x)^d\nonumber\\
&=\alpha_{D,r_0}\vol(B_{1}(0))\cdot\left(\frac{\min_{x\in B_{r(y,\eta)}(y)\cap D}R_n(x)}{k\max_{x\in B_{r(y,\eta)}(y)\cap D}R_n(x)}\right)^d\nonumber\\
&\geq \alpha_{D,r_0}\vol(B_{1}(0))\cdot\frac{1}{k^dG_{\mathcal{N}_n}},
\end{align}
The denominator of the last term in \eqref{as}, on the other hand, can be bounded from above by 
\begin{align}\label{dem}
\sum_{z\in\mathcal N_n} \int_{B_{r(z,\eta)}(z)\cap D} R_n(x)^d \dx{x}&\leq\sum_{z\in\mathcal N_n}\vol(B_{r(z,\eta)}(z)\cap D)\max_{x\in B_{r(y,\eta)}(y)\cap D}R_n(x)^d\\
&\leq\sum_{z\in\mathcal N_n}\vol(B_{r(z,\eta)}(z))\max_{x\in B_{r(y,\eta)}(y)\cap D}R_n(x)^d\nonumber\\
&=\sum_{z\in\mathcal N_n} \vol(B_{1}(0)) \cdot\left(r(z, \eta)\max_{x\in B_{r(y,\eta)}(y)\cap D}R_n(x)\right)^d\nonumber\\
&\stackrel{\eqref{nice}}{=} \vol(B_{1}(0))\cdot\frac{|\mathcal N_n|}{k^d}\nonumber.
\end{align}
Plugging \eqref{num} and \eqref{dem} into \eqref{as} yields
\begin{align*}
\mathrm{Pr}\left[X_m\in B_{r(y,\eta)}(y)\right] &\geq\frac{\alpha_{D,r_0}}{G_{\mathcal{N}_n}|\mathcal N_n|}.
\end{align*}
By a similar reasoning one can show that by taking $M_n= \alpha_{D,r_0}^{-1}(r+1)G_{\mathcal {N}_n}|\mathcal N_n|\log |\mathcal N_n|$, with probability at least $1-|\mathcal N_n|^{-r}\geq 1-N_n^{-r}$, $\left\{X_{m} \right\}_{m=1}^{M_n}$ intersects every ball in the covering of $\mathcal N_n$, completing the proof.
\end{proof}

\section*{Acknowledgements}
We would like to thank the anonymous referees for their very helpful comments which significantly improved the results and presentation of the paper. 
A.~Narayan thanks Norm Levenberg and Sione Ma'u for a careful reading of an early draft, and for providing several comments that greatly improved the quality of the manuscript. 
Y.~Xu thanks Piotr Hajlasz for a helpful conversation about the measure density condition. 
Y.~Xu and A.~Narayan are partially supported by NSF DMS-1848508. This material is based partially upon work supported by the National Science Foundation under Grant No. DMS-1439786 and by the Simons Foundation Grant No. 50736 while A.~Narayan was in residence at the Institute for Computational and Experimental Research in Mathematics in Providence, RI, during the ``Model and dimension reduction in uncertain and dynamic systems" program.

\printbibliography
  
\end{document}